\documentclass[11pt,letterpaper]{amsart}
\title{Random Groups are not n-cubulated}
\author{Zachary Munro}
\address{Burnside Hall, McGill University, 805 Sherbrooke St W, Montreal, Quebec H3A 2K6, Canada}
\address{Faculty of Mathematics, Amado Building, Technion, Haifa, 3200003, Israel}
\email{munrozachary@campus.technion.ac.il}
\date{}

\usepackage{amsmath, amssymb, amsthm} 
\usepackage{enumitem} 
\usepackage{graphicx} 
\graphicspath{ {figures/} }
\usepackage{mathtools} 
\usepackage{MnSymbol}
\usepackage[dvipsnames]{xcolor} 
\usepackage[pagebackref,breaklinks,hidelinks,unicode]{hyperref} 
    \renewcommand*{\backrefalt}[4]{\ifcase #1 (Not cited).\or (Cited p.~#2).\else (Cited pp.~#2).\fi} 
\usepackage{subcaption}
\usepackage{tikz-cd} 

\makeatletter\newcommand\enumlabel[1][]{\item[#1]
    \refstepcounter{enumlabelcount}\def\@currentlabel{#1}}\makeatother

\newenvironment{claim*}{\medskip\noindent\textbf{Claim:}\hspace{0.5mm}}{}

\newenvironment{claim*proof}{\medskip\noindent\emph{Proof of Claim.}\hspace{0.5mm}}
    {\leavevmode\unskip\penalty9999\hbox{}\nobreak\hfill\quad\hbox{$\diamondsuit$}\medskip}

\newcounter{enumlabelcount}
\newcounter{claimcount}

\newtheorem{theorem}{Theorem}[section]
\newtheorem{lemma}[theorem]{Lemma} 
\newtheorem{corollary}[theorem]{Corollary}

\theoremstyle{definition}
\newtheorem{definition}[theorem]{Definition}
\newtheorem{remark}[theorem]{Remark}

\definecolor{note}{rgb}{0.6,0,0.4}

\let\o\H

\newcommand{\A}{\mathcal{A}}
\newcommand{\B}{\mathcal{B}}

\newcommand{\defeq}{\vcentcolon =}
\newcommand{\dist}{\mathsf{d}}
\renewcommand{\H}{\mathcal{H}}
\renewcommand{\L}{\mathcal{L}}
\newcommand{\N}{\mathbf{N}}
\renewcommand{\P}{\mathcal{P}}
\newcommand{\Pcross}{\P^\pitchfork}
\newcommand{\Ppar}{\P^\parallel}
\newcommand{\Psame}{\P^\vert}
\newcommand{\Prob}{\mathbf{P}}
\newcommand{\s}{\mathfrak{s}}
\newcommand{\Z}{\mathbf{Z}}
\newcommand{\RN}[1]{\textbf{\textup{\uppercase\expandafter{\romannumeral#1}}}}

\DeclareMathOperator{\Fix}{Fix_S}
\DeclareMathOperator{\Isom}{Isom}
\DeclareMathOperator{\lk}{Link}
\DeclareMathOperator{\pre}{Prefix}
\DeclareMathOperator{\suf}{Suffix}

\begin{document}

\begin{abstract}
    A group $G$ has property $FW_n$ if every action on an $n$-dimensional $\mathrm{CAT}(0)$ cube complex has a global fixed-point. This provides a natural stratification between Serre's property $FA$ and Kazhdan's property $(T)$. For every $n$, we show that random groups in the plain words density model have $FW_n$ with overwhelming probability. The same result holds for random groups in the reduced words density model assuming there are sufficiently many generators. These are the first examples of cubulated hyperbolic groups with $FW_n$ for $n$ arbitrarily large.
    
    We prove several corollaries: Firstly, we deduce a Rips short exact sequence for $FW_n$ groups, i.e., every finitely presented $FW_n$ group is the quotient of a cubulated hyperbolic group with $FW_n$. Secondly, we construct the first example of a finitely generated group of geometric dimension $2$ which acts freely on a locally finite $\mathrm{CAT}(0)$ cube complex but whose every action on a finite-dimensional $\mathrm{CAT}(0)$ cube complex has a global fixed-point.  
\end{abstract}

\maketitle

\section{Introduction}

The density model of random groups, defined by Gromov in \cite{gromov:asymptoticInvariants}, gives credence to the claim that ``most'' groups are $\delta$-hyperbolic. The model depends upon a parameter $d\in (0,1)$, the \emph{density}, and a parameter $L\in\mathbf{N}$, the \emph{length}. One typically fixes $d$ and studies properties of a random group as $L\to \infty$. A random group at density $d$ is said to satisfy a property $Q$ \emph{with overwhelming probability} if the probability that $Q$ holds goes to $1$ as $L\to \infty$. This is analogous to the study of random graphs in the Erd\o{o}s--Renyi model $G(n,p)$, in which one typically fixes $p$, the probability of adding an edge, and takes the number of vertices $n\to\infty$. More precise definitions are given in Section~\ref{sec:mainTheorem}.

A random group is defined by a presentation, and the relators are randomly selected from a set of words of length $L$. We work mainly in the \emph{plain words model}, in which the relators of the defining presentation are selected from all words of length $L$. In the \emph{reduced words model}, the relators are selected from the reduced words of length $L$.

Random groups in the reduced words model behave vastly differently above and below the critical density $d=\frac12$. Random groups at densities $d<\frac12$ are infinite hyperbolic, but random groups at densities $d>\frac12$ are either trivial or $\Z/2\Z$ \cite{gromov:asymptoticInvariants,ollivier:sharp}. Ollivier proved that a similar phase transition takes place in the plain words model, but with critical density $d=1-\theta$, where $\theta>\frac12$ depends on the number of generators \cite[Theorem~4]{ollivier:sharp}, and $\theta\to \frac12$ as $k\to\infty$. The density model of random groups has proven to be both a fruitful testing ground for conjectures and an important source of examples in group theory. For example, Calegari and Walker were able to give a positive answer to Gromov's surface subgroup question for reduced word random groups \cite{calegari:surfaceSubgroups}. And both the reduced and plain word models provide many examples of non-left-orderable and property (T) groups \cite{orlef:randomNotOrderable},\cite{kotowskikotowski:randomGroupsPropertyT,zuk:propertyT}.

We study the actions of random groups on $\mathrm{CAT}(0)$ cube complexes. Despite the increasing popularity of $\mathrm{CAT}(0)$ cube complexes in the past few decades, the study of fixed-point properties for actions on $\mathrm{CAT}(0)$ cube complexes has been relatively neglected. We will take particular interest in the following fixed-point property for actions on finite-dimensional $\mathrm{CAT}(0)$ cube complexes.

\begin{definition}[$FW_n$, $FW_\infty$, $FW$]
    A group $G$ has property $FW_n$ if for every action of $G$ on an $n$-dimensional $\mathrm{CAT}(0)$ cube complex, the fixed-point set of $G$ is non-empty. A group $G$ has $FW_\infty$ if it has $FW_n$ for every $n$. we say $G$ has $FW$ if the fixed-point set of every action of $G$ on a $\mathrm{CAT}(0)$ cube complex is non-empty.
\end{definition}

For a recent discussion on known results regarding $FW_n$ along with many open questions, see \cite{genevois:FWn}. 

Property $FW_1$ is exactly Serre's property $FA$, and property $(T)$ is stronger than any of the above fixed-point properties. The properties $FW_n$ can be thought of as a stratification between the well-studied properties $FA$ and $(T)$. That is, one has the following inclusions 
\[
    FA=FW_1\supset FW_2\supset\cdots\supset FW_\infty\supset FW\supset (T).
\]

At density $d>1/3$, random groups with either plain or reduced words have property (T) \cite{kotowskikotowski:randomGroupsPropertyT, zuk:propertyT}. By a result of Niblo and Reeves \cite{nibloreeves:groupsActingCat0CubeComplexes}, a group with property (T) has property $FW_\infty$, thus at density $d>1/3$ every action on a finite-dimensional $\mathrm{CAT}(0)$ cube complex has a global fixed-point. However, at lower densities this changes. Ashcroft proved that random groups in the reduced words model at density $d<1/4$ act without global fixed-point on finite-dimensional $\mathrm{CAT}(0)$ cube complexes \cite{ashcroft:randomGroups1/4}. Even more strongly, random groups with reduced words at density $d<1/12$ satisfy the $C'(1/6)$ small-cancellation condition \cite[§9.B]{gromov:asymptoticInvariants}, and Wise proved that $C'(1/6)$ groups act properly and cocompactly on $\mathrm{CAT}(0)$ cube complexes \cite{wise:cubulatingC'(1/6)}. Ollivier and Wise strengthened this result to include random groups in the reduced words model at density $d<1/6$ \cite{ollivierwise:randomGroups1/6}. The above results contrast sharply with the main theorem of the article, stated below.


\begin{theorem}
\label{thm:mainTheorem}
    For any $d\in(0,1)$ and $n\in \mathbf N$, a random group in the plain words density model at density $d$ has property $FW_n$ with overwhelming probability.
\end{theorem}

In a sense, this is the strongest result one could hope for, since it holds for any density $d\in (0,1)$ and any dimension $n\in\mathbf{N}$. As mentioned above, the statement was previously known to hold for $d>\frac13$ and any $n\in\mathbf{N}$ by invoking property (T). However, our result circumvents the use of property (T). For statements concerning all densities $d\in(0,1)$, the theorem was only known to hold in the cases $n=1$ \cite{dahmaniguirardelprzytycki:randomDontSplit} and $n=2$ \cite{munro:randomCAT(0)SquareComplex}. In the context of small-cancellation, Jankiewicz constructed, for each $p\geq 6$ and $n\in \mathbf N$, examples of $C'(1/p)$ groups which do not act properly on any $n$-dimensional $\mathrm{CAT}(0)$ cube complex \cite{jankiewicz:lowerBounds}. 

Theorem~\ref{thm:mainTheorem} is proven by finding generic conditions which imply $FW_n$. This is analogous to the situation in \cite{zuk:propertyT}, where \.Zuk's criterion functions as a generic condition implying (T). And just as \.Zuk's criterion functions as a flexible means to create examples of groups with (T), we hope our conditions will enable the creation of interesting $FW_n$ groups. Parts of the proof of Theorem~\ref{thm:mainTheorem} are inspired by \cite{jahncke:randomGroupsPropertyFR, dahmaniguirardelprzytycki:randomDontSplit}. In particular, for each $n\in \mathbf N$, we construct a finite family of automata whose accepted languages force a group to act with global fixed point on any $n$-dimensional $\mathrm{CAT}(0)$ cube complex. See Section~\ref{sec:ideaOfProof} below for more details. By modifying our automata to accept only reduced words, we can prove a version of Theorem~\ref{thm:mainTheorem} in the reduced words density model, supposing there are sufficiently many generators.

\begin{theorem}
\label{thm:reducedWords}
    For every $n\in \mathbf N$, there exists $k(n)>0$ such that the following holds: For any $d\in(0,1)$, a random group in the reduced words model with density $d$ and $k\geq k(n)$ generators has property $FW_n$ with overwhelming probability.  
\end{theorem}

As a corollary of Theorem~\ref{thm:reducedWords}, we are able to strengthen the main theorem of \cite{jankiewicz:lowerBounds}, mentioned above.  In fact, for $p\geq 6$, we are able to show that a strong version of the theorem holds for generic $C'(1/p)$ groups as relator length tends to infinity. Note that $FW_n$ implies a group cannot act properly on an $n$-dimensional $\mathrm{CAT}(0)$ cube complex.

\begin{corollary}
\label{cor:cubicalDimension}
    Fix $n\in \mathbf N$ and $0<d<\frac{1}{2p}$ with $p\geq 6$. Let $k(n)$ be as in Theorem~\ref{thm:reducedWords}. Let $G$ be a random group in the reduced words density model at density $d$ with $k\geq k(n)$ generators. Then with overwhelming probability, $G$ is $C'(1/p)$ and has $FW_n$. In particular, for each $n$, there exist $C'(1/p)$ groups with no proper actions on $n$-dimensional $\mathrm{CAT}(0)$ cube complexes.
\end{corollary}

\begin{proof}
    It is a theorem of Gromov that $G$ is $C'(1/p)$ with overwhelming probability \cite[§9.B]{gromov:asymptoticInvariants}. See also \cite[Proposition~1.8]{ollivierwise:randomGroups1/6}. By Theorem~\ref{thm:reducedWords}, $G$ has property $FW_n$ with overwhelming probability. 
\end{proof}

Genevois constructs cubulated groups with $FW_n$ for arbitrarily large $n$ \cite{genevois:FWn}. However, the groups he constructs are virtually free abelian. Corollary~\ref{cor:cubicalDimension} provides the first examples of cubulated hyperbolic groups with $FW_n$ for arbitrarily large $n$.

Theorem~\ref{thm:mainTheorem} and Theorem~\ref{thm:reducedWords} produce many examples of groups with $FW_n$ but not $FW_\infty$. To produce a group with $FW_\infty$, Corollary~\ref{cor:cubicalDimension} can be utilized in a construction similar to that in \cite{jankiewiczWise:curiouslyCubulated} to form an infinite $C'(1/6)$ complex $Z$ with $FW_\infty$ fundamental group. Thus $\pi_1 Z$ acts freely on a locally finite $\mathrm{CAT}(0)$ cube complex, but every action of $\pi_1 Z$ on a finite-dimensional $\mathrm{CAT}(0)$ cube complex has a global fixed-point. Previously, it was known that Thompson's group $V$ has $FW_\infty$ \cite{genevois:thompsonV} yet acts freely on a locally finite $\mathrm{CAT}(0)$ cube complex \cite{farley:pictureGroups}. However, the example below is the first with cohomological dimension $2$.

\begin{theorem}
\label{thm:aMonster}
    There exists a finitely generated $FW_\infty$ group with cohomological dimension $2$ acting freely on a locally finite $\mathrm{CAT}(0)$ cube complex.
\end{theorem}

Given a finitely presented group $G=\langle S\mid R\rangle$, the \emph{Rips construction} produces another finitely presented group $\Gamma=\langle S\cup\{x,y\}\mid R'\rangle$ satisfying $C'(1/p)$ so that there exists a short exact sequence $1\to \langle x,y\rangle\to \Gamma\to G\to 1$ \cite{rips:subgroupsSmallCancellation}. If $G$ has $FW_n$, then Theorem~\ref{thm:reducedWords} can be used to perform a variation of the Rips construction so that $\Gamma$ is $C'(1/p)$ and has $FW_n$. More precisely, we have the following.

\begin{corollary}[$FW_n$ Rips Construction]
\label{cor:ripsConstruction}
    Let $G$ be a finitely presented group with property $FW_n$. For any $p\geq 6$, $G$ fits into a short exact sequence of the form
    \[
    1\to K\to \Gamma\to G\to 1,
    \]
    where $K$ is a finitely generated $FW_n$ group and $\Gamma$ is a $C'(1/p)$ group with $FW_n$. 
\end{corollary}

To prove Corollary~\ref{cor:ripsConstruction}, we establish a a combination result for short exact sequences of $FW_n$ groups which may be of independent interest. 

\begin{lemma}
\label{lem:shortExactSequence}
    If $K$ and $N$ have property $FW_n$, and $G$ fits into a short exact sequence of the form
    \[
    1\to K\to G\to N\to 1,
    \]
    then $G$ has property $FW_n$. 
\end{lemma}

\subsection{Idea of proof}
\label{sec:ideaOfProof}
The central idea of this work is to construct, for any dimension $n>0$ and finite set of generators $S$, a finite set of automata $\mathbf \Sigma$ such that the following holds: Let $F_S$, the free group over $S$, act on an $n$-dimensional $\mathrm{CAT}(0)$ cube complex $X$. There exists a point $x \in X^0$ and an automaton $\Sigma\in\mathbf\Sigma$ where, for every word $w$ in the accepted language of $\Sigma$, $wx \neq x$. The automata in $\mathbf\Sigma$ also have large growth, meaning roughly that the number of accepted words of length $L>0$ grows exponentially as a function of $L$. Importantly, the set of automata $\mathbf\Sigma$ is finite and does not depend on the specific choice of $\mathrm{CAT}(0)$ cube complex $X$ or the action of $F_S$ on $X$.  

Consider a group $G = \langle S \mid R \rangle$, where $R$ includes a word accepted by $\Sigma$ for each $\Sigma\in \mathbf\Sigma$. If $G$ acts without a global fixed-point on an $n$-dimensional $\mathrm{CAT}(0)$ cube complex $X$, then there is an induced action of $F_{S}$ on $X$ without a global fixed-point. This leads to a contradiction, since $R$ contains a word accepted by each of the automata in $\mathbf\Sigma$. Hence $G$ has $FW_n$. Random groups are defined by selecting a random set of relators. Since the accepted languages of $\Sigma\in \mathbf\Sigma$ have large growth, a random set of relators intersects each language with high probability. Hence for each $n>0$, random groups have $FW_n$.

The finite set of automata $\mathbf{\Sigma}$ is noteworthy on its own, as it allows for the construction of $FW_n$ groups beyond just random groups. To make it more broadly applicable, we have described the set $\mathbf\Sigma$ in Theorem~\ref{thm:constructingExamples} in a way that allows for easy adaptation to other contexts. Most of the paper is dedicated to constructing the automata in $\mathbf\Sigma$. This requires a careful study of the local arrangement of hyperplanes about a point $x\in X^0$ and their translates. Some of the ideas involved are discussed in Section~\ref{sec:structureOfPaper} below.

\subsection{Structure of the paper}
\label{sec:structureOfPaper}
In Section~2, we establish the necessary preliminaries for the paper. We introduce notation and basic definitions related to formal words and $\mathrm{CAT}(0)$ cube complexes. Additionally, we provide background on hyperplanes and their fundamental properties. These are subspaces of a $\mathrm{CAT}(0)$ cube complex which determine much of its geometry.

In Section~3, we define and analyze \emph{progressing automata} (Definition~\ref{def:progressingAutomaton}), beginning with a review of standard terminology from graph theory and automata. Broadly, an automaton $\Sigma$ accepts a set of words $w$ that label specific directed paths $P \to \Sigma$. Progressing automata incorporate additional structure. They are defined relative to an action of the free group $F_S$ on a $\mathrm{CAT}(0)$ cube complex $X$, with a designated basepoint $x \in X^0$. A key feature of progressing automata is the presence of \emph{checkpoint vertices}, each labeled with a hyperplane near $x$. Intuitively, the defining property of a progressing automaton $\Sigma$ is that if a directed path $P \to \Sigma$ labeled by $w$ terminates at a checkpoint vertex, then a translate of the corresponding hyperplane separates $x$ from $wx$. In particular, every accepted word $w$ satisfies $wx \neq x$. After proving fundamental properties of progressing automata, we introduce \emph{checkpoint trees} (Definition~\ref{def:checkpointTree}), which serve as essential components in constructing progressing automata. The section concludes with a simple case not requiring the machinery of progressing automata. We construct ``by hand'' an automaton whose accepted language moves the basepoint $x$ (Lemma~\ref{lem:generatorsWithFixedPoint}).

Section~4 focuses on the construction of checkpoint trees, involving a detailed analysis of various local hyperplane arrangements around the basepoint $x$. We introduce a partition of $S^\pm=S\cup S^{-1}$, the generators of $F_S$ and their inverses, which underpins our constructions (Definition~\ref{def:ABP}). Given a hyperplane $H$ closest to $x \in X^0$, we partition the generating set $S^\pm = S \cup S^{-1}$ into three subsets: $\mathcal{A}(H)$, $\mathcal{B}(H)$, and $\mathcal{P}(H)$. Specifically, elements of $\mathcal{A}(H)$ move $x$ away from $H$ without crossing it, elements of $\mathcal{B}(H)$ move $x$ across $H$, and elements of $\mathcal{P}(H)$ move $x$ parallel to $H$. The subset $\mathcal{P}(H)$ is further divided into $\Psame(H)$, $\Ppar(H)$, and $\Pcross(H)$, each classified according to how they act on the hyperplane $H$. The construction of progressing automata is carried out on a case-by-case basis, depending on the distribution of elements among these subsets. The main result of this section (Theorem~\ref{thm:mainTechnicalTheorem}) states that if a sufficiently large subset of $S^\pm$ does not fix $x$ pointwise, then there exists a progressing automaton whose size is uniformly bounded by a constant depending only on the dimension of $X$.

In Section~5, we leverage the existence of bounded-size progressing automata to prove Theorem~\ref{thm:mainTheorem}. We begin by introducing random groups and stating the key fact about intersections of high density languages (Lemma~\ref{lem:denseLanguageIntersection}). Next, we summarize relevant properties of $FW_n$. To establish Theorem~\ref{thm:mainTheorem}, we consider a finite-index subgroup $H_m$ of a random group and a related group $\widehat{G}_m$ that quotients onto $H_m$. We demonstrate that the relator set of $\widehat{G}_m$ intersects the languages of a finite family of progressing automata $\mathbf{\Sigma}$ with overwhelming probability, which implies that $\widehat{G}_m$ satisfies $FW_n$. Since property $FW_n$ is preserved under quotients and finite-index supergroups, the theorem follows. In the last couple subsections, we complete the proofs of the remaining results, including Theorems~\ref{thm:reducedWords} and \ref{thm:aMonster}, as well as Corollary~\ref{cor:ripsConstruction}.

\subsection{Further directions}
Here we collect some questions arising from the work in this paper.

\begin{enumerate}
    \item We show that random groups have $FW_n$ with overwhelming probability, which means $\Prob\{G\text{ has }FW_n\}\to 1$ as $L\to \infty$. For a fixed $n$ and density $d\in(0,\frac13)$, at what rate does the probability converge to $1$?
    \item What about other models of random groups, e.g., the few relators model? 
    \item Does there exists a $C'(1/6)$ group with $FW_n$ that acts properly cocompactly on an $(n+1)$-dimensional $\mathrm{CAT}(0)$ cube complex?
\end{enumerate}

This last question is related to Problem~4.3 and 4.4 in \cite{genevois:FWn}.


\medskip
\textbf{Acknowledgements.} Thank you, Piotr Przytycki, for posing the question of whether random groups have a cubical fixed-point property and for the many insightful conversations we had throughout my work on this problem. Your knowledge and support were invaluable. Thanks also to Dani Wise for sharing his deep understanding and enthusiasm for cubes. I am grateful to Harry Petyt for his encouragement---and for motivating me to get off the couch and complete this article. I thank the anonymous referees for their careful reading and edits, which greatly improved the clarity of this article.


\section{Preliminaries}

\subsection{Words, subwords, and the free group}
\label{sec:wordsSubwords}

Let $S=\{s_1,\ldots, s_k\}$ be a set of formal letters, $S^{-1}=\{s_1^{-1},\ldots, s_k^{-1}\}$ the set of formal inverses, and $S^\pm=S\sqcup S^{-1}$ their union. A \emph{word $w$ over $S^\pm$} is a finite sequence of elements from $S^\pm$. The elements of the sequence are the \emph{letters} of $w$. The \emph{length} of a word $w$ is the length of the sequence, denoted $|w|$. The \emph{empty word} is the unique word of length zero. Apart from the empty word, each word $w$ can be denoted $w=w(1)\cdots w(|w|)$ where $w(i)\in S^\pm$ is the $i$-th letter of $w$.  A \emph{subword} $w'$ of $w$ is a subsequence of $w$. That is, there exists $j\geq 0$ so that $w'(i)=w(j+i)$ for $i\in \{1,\ldots, |w'|\}$. The subword $w'$ is a \emph{prefix} if $j=0$, and it is a \emph{suffix} if $j=|w|-|w'|$. We say $w'$ is a \emph{proper subword} if $|w'|<|w|$. The empty word is a prefix and a suffix of every word. The set of prefixes and suffixes of $w$ are denoted $\pre(w)$ and $\suf(w)$. Thus the proper prefixes and suffixes are exactly $\pre(w)-\{w\}$ and $\suf(w)-\{w\}$. Given words $w$ and $u$, their \emph{concatenation} $wu$ is the word formed by concatenating sequences. More formally, $wu$ is the word of length $|w|+|u|$ defined by $wu(i)=w(i)$ for $i\in\{1,\ldots, |w|\}$ and $wu(i)=u(i-|w|)$ for $i\in\{|w|+1,\ldots, |w|+|u|\}$. A \emph{language} is a set of words, often denoted by $\mathcal L$. For a word $r$ and language $\mathcal L$, we define $r\mathcal L=\{rw :w\in \mathcal L\}$ and $\mathcal L r=\{wr: w\in \mathcal L\}$.

When dealing with the free group $F_S$ over $S$, we will sometimes view words over $S^\pm$ as elements of $F_S$. It should always be clear from context whether we are referring to a word just as a sequence of letters or as an element of $F_S$.

\subsection{$\mathrm{CAT}(0)$ cube complexes}

We review some standard definitions and results concerning $\mathrm{CAT}(0)$ cube complexes. For a more detailed introduction, see \cite{sageev:cat0CubeComplexes, wise:structureOfGroups}. 

An \emph{$n$-cube} is a copy of $[-1,1]^n$. A \emph{face} of a cube is a subspace defined by restricting some coordinates to $1$ or $-1$. A face naturally has the structure of a cube. For example, the $2$-cube $[-1,1]^2$ has four $1$-cube faces: $\{-1\}\times [-1,1]$, $\{1\}\times [-1,1]$, $[-1,1]\times \{-1\}$, and $[-1,1]\times \{1\}$. The faces of a cube give it a combinatorial cell structure where the closed $n$-cells are the $n$-faces. A \emph{cube complex} is constructed from a collection of cubes by identifying some faces via \emph{combinatorial isomorphisms}, i.e., isomorphisms sending $n$-cells homeomorphically to $n$-cells. The \emph{dimension} of a cube complex is the maximal dimension of a cube in the collection. For example, a $1$-dimensional cube complex is a graph. Similarly, a \emph{simplex complex} is a complex built from simplicies by identifying some faces via combinatorial isomorphisms. Simplicial complexes are exactly those simplex complexes whose simplicies are uniquely identified by their $0$-cells. A simplicial complex is \emph{flag} if each set of pairwise adjacent $0$-simplicies spans a simplex. 

A corner of an $n$-cube can be naturally associated to an $(n-1)$-simplex where the ends of $1$-cubes at the corner correspond to the 0-cells of the simplex. Under this correspondence, codimension-$k$ faces of the simplex are in bijection with the codimension-$k$ faces of the cube at the corner. Let $X$ be a cube complex. The \emph{link} $\lk(x)$ of a $0$-cube $x\in X^0$ is the simplex complex with an $(n-1)$-simplex for each corner of an $n$-cube containing $x$. Faces of simplicies are identified if the corresponding faces in the corners of cubes are identified. 

A cube complex $X$ is \emph{non-positively curved} if $\lk(x)$ is a flag simplicial complex for each $0$-cube $x\in X^0$. A \emph{CAT(0) cube complex} is a simply-connected non-positively curved cube complex. It is worth mentioning that $\mathrm{CAT}(0)$ is a potential property of any metric space. The condition on $\lk(x)$ for each $x\in X^0$ guarantees $X$ is $\mathrm{CAT}(0)$ when given the $\ell^2$-metric \cite{gromov:hyperbolicGroups,leary:metricKanThurstonTheorem}, i.e., the path metric on $X$ where each cube has been given the $\ell^2$-metric. 

A \emph{midcube} of a cube is the subspace defined by restricting one of the coordinates to 0, e.g., the $2$-cube has two midcubes, $\{0\}\times [-1,1]$ and $[-1,1]\times \{0\}$. Midcubes naturally have the structure of a cube. Let $\Upsilon(X)$ be the cube complex formed from copies of midcubes of $X$ glued along common midcubes. There is a natural immersion, i.e., a locally injective map, $\Upsilon(X)\looparrowright X$ given by the inclusion map on each midcube. A \emph{hyperplane} $H\looparrowright X$ is the restriction of $\Upsilon(X)\looparrowright X$ to a connected component $H$ of $\Upsilon(X)$. The inclusion of a midcube in a cube $m\hookrightarrow C$ can be extended to an isomorphism $m\times [-1,1]\to C$. Extending each inclusion of a midcube in this fashion defines an immersion $N(H)\looparrowright X$ called the \emph{carrier} of the hyperplane where $N(H)= H\times [-1,1]$.

The following theorem is standard in the study of $\mathrm{CAT}(0)$ cube complexes. For proofs see \cite[Theorem~1.1]{sageev:cat0CubeComplexes} or \cite[Lemma~2.14]{wise:structureOfGroups}.

\begin{theorem}
\label{thm:hyperplanes}
    Let $X$ be a $\mathrm{CAT}(0)$ cube complex, and let $H$ be a hyperplane of $X$.
        \begin{enumerate}
            \item The map $N(H)\looparrowright X$ is an embedding. 
            \item The image $N(H)\subset X$ is a convex subcomplex.
            \item $H$ and $N(H)$ are $\mathrm{CAT}(0)$ cube complexes.
            \item Every collection of pairwise intersecting hyperplanes has non-empty intersection.
        \end{enumerate}
\end{theorem}

In light of Theorem~\ref{thm:hyperplanes}(1), we can view carriers $N(H)\looparrowright X$ as subspaces $N(H)\subset X$. Similarly, a hyperplane $H\looparrowright X$ can be viewed as a subspace $H\subset X$. Hyperplanes are already implicitly viewed as subspaces in the statement of Theorem~\ref{thm:hyperplanes}(4). Since a midcube is two-sided, i.e., a neighborhood of $m\subset C$ is homeomorphic to $m\times (-\epsilon,\epsilon)$, hyperplanes in a $\mathrm{CAT}(0)$ cube complex are two-sided. Thus the complemenent of a hyperplane $X-H$ has two components. The two components of $X-H$ are the \emph{halfspaces} of $H$. Additionally, if a $0$-cube $x$ lies in a cube intersecting $H$, then $x\in N(H)$ and there exists a unique $y\in N(H)$ such that $x$ and $y$ lie in different halfspaces of $H$ and are joined by a $1$-cube. This can be seen from the product structure $N(H)=H\times [-1,1]$, and we say $y$ is \emph{opposite to $x$ across $H$}. If $k$ hyperplanes pairwise intersect, then there exists a $k$-cube whose midcubes correspond to each of the $k$ hyperplanes. Thus, in an $n$-dimensional $\mathrm{CAT}(0)$ cube complex no more than $n$ hyperplanes can pairwise intersect. 

We will work mainly with the \emph{graph metric} on the $0$-cubes $X^0$, which we denote $(X^0,\dist)$. See Section~\ref{sec:graphs} below for background on graph theory. Treating the $1$-skeleton $X^1$ as a graph, $\dist(x,y)$ is the infimal length of a path $P\to X^1$ joining $x$ and $y$. A path $P\to X^1$ is \emph{geodesic} if it is a shortest length path joining its endpoints. We say a subcomplex $A\subset X$ is \emph{convex} if every geodesic joining $0$-cubes of $A$ lies in $A$. We define $\dist(A,B)= \inf_{a\in A^0, b\in B^0}\dist(a,b)$ for convex subcomplexes $A,B\subset X$. The graph metric is determined by the hyperplanes of $X$ in a way we now make precise. A hyperplane $H$ \emph{separates} subsets $A,B\subset X$ if $A$ and $B$ lie in distinct halfspaces of $H$. For subsets $A,B\subset X$, we let $\#(A,B)$ denote the number of hyperplanes separating $A$ and $B$. A hyperplane $H$ is \emph{closest} to $A\subset X$ if there does not exist a hyperplane separating $A$ and $H$. The following lemma follows from \cite[Corollary~2.15]{wise:structureOfGroups} and \cite[Lemma~2.18]{wise:structureOfGroups}.   

\begin{lemma}
\label{lem:hyperplanes}
    For convex subcomplexes $A,B\subset X$, $\dist(A,B)=\#(A,B)$. 
    
    In particular, $\dist(x,y)=\#(x,y)$ for $0$-cubes $x,y\in X$. And for a hyperplane $H$ and $0$-cube $x\in X^0$, we have $x\not\in N(H)$ if and only if there exists a hyperplane $H'$ separating $x$ and $N(H)$. 
\end{lemma}

A consequence of Lemma~\ref{lem:hyperplanes} is that for a convex subcomplex $A\subset X$, a hyperplane $H$ is closest to $A$ if and only if $A\cap N(H)\neq \emptyset$. We conclude this section with some definitions regarding the arrangement of hyperplanes. Hyperplanes $H$, $H'$ of $X$ \emph{cross} if $H\neq H'$ and $H\cap H'\neq \emptyset$. This is denoted $H\pitchfork H'$. Two hyperplanes $H$, $H'$ are \emph{parallel} if $H\cap H'=\emptyset$. This is denoted $H\parallel H'$. If $H$ and $H'$ are non-crossing, it is still possible that $H\cap H'\neq \emptyset$. This happens exactly when $H=H'$.

\section{Progressing Automata} 

In this section we introduce progressing automata and establish their key properties. We also define checkpoint trees and show how they can be used to construct progressing automata. We begin by reviewing definitions and terminology in graph theory.

\subsection{Graphs, paths, and trees}
\label{sec:graphs}

We review some standard graph theory terminology that will be used throughout the paper. A \emph{graph} $\Sigma$ is a 1-complex whose 0-cells we call \emph{vertices} and whose 1-cells we call \emph{edges}. We sometimes write $\Sigma=(V,E)$, where $V$ denotes the vertex set, and $E$ denotes the edge set. Morphisms of graphs are combinatorial maps, i.e., they send vertices to vertices and edges to edges. A \emph{directed edge} is an edge with a fixed orientation, and a \emph{directed graph} is a graph whose edges are all directed. An orientation can be represented by an arrow on the edge, allowing us to speak of the edge being directed \emph{towards} or \emph{away} from an endpoint. If $v$ is a vertex of a directed graph, then the \emph{outgoing edges at $v$} are the edges with an endpoint at $v$ that are directed away from this endpoint. The \emph{incoming edges at $v$} are the edges with an endpoint at $v$ that are directed towards this endpoint. A directed edge with both endpoints equal to $v$ is both an outgoing and incoming edge at $v$. 
    
A \emph{path in $\Sigma$} is a map $P\to \Sigma$ where $P$ is a graph identified with $[0,n]$, where $[0,n]$ has a $0$-cell at each integer point. A path is \emph{trivial} if $n=0$. The edges of a non-trivial path $P\to \Sigma$ are naturally ordered, and there exists an edge decomposition $P=e_1\cdots e_n$ where $e_i=[i-1,i]$. The edges of $P$ are also naturally directed, with each edge directed towards its larger endpoint. The \emph{initial vertex} of $P$ is the image of $0$ in $P\to \Sigma$, and the \emph{terminal vertex} is the image of $n$. If $\Sigma$ is directed, then a path $P\to \Sigma$ is \emph{directed} if it respects orientations of edges. 
    
A \emph{vertex/edge labeling over a set $\mathcal A$} associates to each vertex/edge of a graph an element of $\mathcal A$. Formally, a labeling over $\mathcal A$ is a map from the set of vertices/edges to $\mathcal A$. We also allow labelings on subsets of vertices and edges. An edge labeling of a directed graph is \emph{deterministic} if the outgoing edges at each vertex have distinct labels. Our label set will most often be $S^\pm$. An edge labeling of $\Sigma$ pulls back to an edge labeling of any path $P\to \Sigma$. If the edges $e_1,\ldots, e_n$ of $P$ are labeled by letters $w(1),\ldots, w(n)$ in $S^\pm$, then the \emph{label of $P$} is the word $w=w(1)\cdots w(n)$.
    
A \emph{tree} is a graph $\Sigma$ with $\pi_1\Sigma$ trivial. A \emph{leaf} of a tree is a vertex which is the endpoint of a single edge. A \emph{rooted tree} $\Sigma$ is a tree with a distinguished vertex, the \emph{root}. The \emph{depth} of a vertex $v\in \Sigma$ is its distance from the root, and the \emph{depth} of $\Sigma$ is the maximal depth across all vertices $v\in \Sigma$. The root has depth zero. A path $P\to \Sigma$ is \emph{rooted} if the initial vertex of $P$ is the root of $\Sigma$. As a convention, we orient edges of a rooted tree away from the root. That is, each edge is directed towards its endpoint farthest from the root. Let $\Sigma$ be a rooted tree. A vertex $v$ is a \emph{descendant} of $u$ if there exists a non-trivial directed path $P\to \Sigma$ with initial vertex $u$ and terminal vertex $v$. If $P$ has length one, then $v$ is a \emph{child} of $u$. If $v$ is a descendant of $u$, then $u$ is an \emph{ancestor} of $v$.

\subsection{Progressing automata and checkpoint trees} 

We now move towards defining progressing automata.


\begin{definition}[Finite state automaton]
    A \emph{finite state automaton over $S^\pm$}, or simply an \emph{automaton over $S^\pm$}, is a finite directed graph $\Sigma$ with an edge labeling over $S^\pm$  and a distinguished \emph{start vertex} $\s$. A word $w$ is \emph{accepted} by $\Sigma$ if it is the label of a directed path $P\to \Sigma$ beginning at $\s$. The \emph{accepted language} $\L_\Sigma$ is the set of all accepted words. 
\end{definition}

In general, the accepted language $\L_\Sigma$ of an automaton $\Sigma$ may not contain many words. The next definition is a property of $\Sigma$ which guarantees that the accepted language $\L_\Sigma$ is ``dense''. 

\begin{definition}[$\lambda$-large automaton]
\label{def:largeGrowth}
    Let $\lambda\in [0,1]$. A vertex $v$ of a directed graph deterministically labeled over $S^\pm$ has \emph{$\lambda$-large growth} if there are at least $\lambda|S^\pm|$ outgoing edges at $v$. An automaton has \emph{$\lambda$-large growth} if there exists some $K>0$ so that each directed path $P\to \Sigma$ beginning at $\s$ and of length at least $K$ terminates on a vertex with $\lambda$-large growth.
\end{definition}
    
\begin{definition}[Visible hyperplanes]
\label{def:visibleHyperplanes}
    Fix an action of $F_S$ on a $\mathrm{CAT}(0)$ cube complex $X$, and fix a \emph{basepoint} $0$-cube $x\in X^0$. For $s\in S^\pm$, let $\H_s$ be the set of hyperplanes separating $x$, $sx$ closest to $x$. The \emph{visible} hyperplanes are $\H=\cup_{s\in S^\pm}\H_s$. 
\end{definition}

By Lemma~\ref{lem:hyperplanes}, $\H_s$ can equivalently be defined as the set of hyperplanes $H$ separating $x$, $sx$ such that $x\in N(H)$. Note that the set of visible hyperplanes depends on the action of $F_S$ on $X$ and the choice of basepoint. Also note that $\mathcal H$ is finite even if $X$ is infinite-dimensional or locally infinite. This is because finitely many hyperplanes separate $x$, $sx$ for each $s\in S^\pm$, and $S^\pm$ is finite.

\begin{figure}[h]
    \centering
    \begin{subfigure}[h]{0.3\textwidth}
        \centering
        \includegraphics[width=\textwidth]{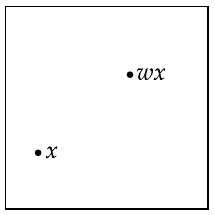}
        \caption{$P$ meets no checkpoint vertices.}	
        \label{}
    \end{subfigure}
	\ \ \ \ \ 
    \begin{subfigure}[h]{0.3\textwidth}
        \centering	
	\includegraphics[width=\textwidth]{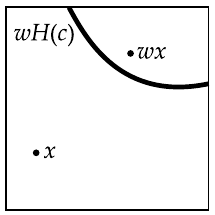}
        \caption{$P$ terminates on a checkpoint vertex $c$. }	
        \label{}
    \end{subfigure}
    \caption{$P\to \Sigma$ begins at $\s$.}
    \label{fig:prog1}
\end{figure}

\begin{figure}[h]
    \centering
    \begin{subfigure}[h]{0.3\textwidth}
        \centering
        \includegraphics[width=\textwidth]{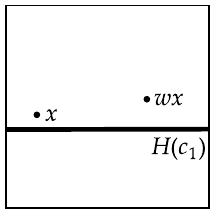}
        \caption{$P$ meets no other checkpoint vertices.}	
        \label{}
    \end{subfigure}
	\ \ \ \ \ 
    \begin{subfigure}[h]{0.3\textwidth}
        \centering	
	\includegraphics[width=\textwidth]{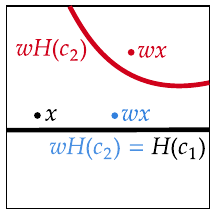}
        \caption{$P$ terminates on a checkpoint vertex $c_2$.}	
        \label{}
    \end{subfigure}
    \caption{$P\to \Sigma$ begins at a checkpoint vertex $c_1$. The red and blue in (B) denote two possible cases.}
    \label{fig:prog2}
\end{figure}

\begin{definition}[Progressing automaton]   
\label{def:progressingAutomaton}
    Adopt the set-up in Definition~\ref{def:visibleHyperplanes}. An automaton $\Sigma=(V,E)$ is \emph{progressing} if there exists a non-empty set of \emph{checkpoint} vertices $C\subset V-\{\s\}$ and a labeling $H:C\to \H$ over visible hyperplanes such that the following hold: 
    
    \begin{enumerate}
        \item  Let $P\to \Sigma$ be a directed path beginning at $\s$ with label $w$ such that either $P$ does not meet $C$, or $P$ meets $C$ only at its terminal vertex. Then $w x\neq x$. If the terminal vertex of $P$ is $c\in C$, then $w H(c)$ separates $x$ and $w x$. See Figure~\ref{fig:prog1}.
        \item Let $P\to \Sigma$ be a directed path with label $w$ such that either $P$ meets $C$ only at its initial vertex, or $P$ meets $C$ only at its initial and terminal vertices. Let $c_1\in C$ be the initial vertex of $P$. The points $x$ and $w x$ lie in the same halfspace of $H(c_1)$. If the terminal vertex of $P$ is $c_2\in C$, then either $w H(c_2)=H(c_1)$, or $w H(c_2)$ is parallel to $H(c_1)$ and separates $x$ and $w x$. See Figure~\ref{fig:prog2}.
    \end{enumerate}
\end{definition}

Note that implicit in the structure of a progressing automaton are a $\mathrm{CAT}(0)$ cube complex $X$, an action of $F_S$ on $X$, and a basepoint $x\in X^0$.

\begin{lemma}
\label{lem:checkpoints}
    Let $\Sigma$ be a progressing automaton, and let $x\in X^0$ be the basepoint. For each $w\in \L_\Sigma$, we have $wx\neq x$. That is, if $P\to \Sigma$ is a directed path beginning at $\s$ with label $w$, then $wx\neq x$. Moreover, if $c\in C$ is the last checkpoint vertex met by $P\to \Sigma$ and $w'$ labels the subpath from $\s$ to $c$, then $w'H(c)$ separates $x$ and $wx$. 
\end{lemma}

\begin{proof}
    We prove the lemma by induction on the number of checkpiont vertices met by $P$. If $P$ meets no checkpoint vertices or has only its terminal vertex in $C$, then the lemma follows from Definition~\ref{def:progressingAutomaton}(1). Now suppose $P$ meets $k\geq 1$ checkpoint vertex, and suppose the lemma holds for all paths meeting fewer checkpoint vertices. Let $c\in C$ be the last checkpoint vertex met by $P$. We consider two separate cases. 
    
    First suppose $c$ is the terminal vertex of $P$. If $c$ is the only checkpoint vertex, then we are in the base case discussed above, so suppose there exists a checkpoint vertex $c_1$ that $P$ meets just before $c_2=c$. Let $P=P_1P_2$ be the decomposition such that $P_1$ terminates on $c_1$ and $P_2$ begins on $c_1$ and terminates on $c_2$. Let $w=w_1w_2$ be the corresponding decomposition of the label. By the inductive hypothesis, $w_1H(c_1)$ separates $x$ and $w_1x$. By the first part of Definition~\ref{def:progressingAutomaton}(2), $x$ and $w_2x$ lie in the same halfspace of $H(c_1)$. Multiplying by $w_1$, we get that $w_1x$, $wx$ lie in the same halfspace of $w_1H(c_1)$. Thus $w_1H(c_1)$ separates $x$, $wx$. This already establishes $wx\neq x$. By the second part of Definition~\ref{def:progressingAutomaton}(2) either $w_2H(c_2)=H(c_1)$, or $w_2H(c_2)$ is parallel to $H(c_1)$ and separates $x$, $w_2x$. If $w_2H(c_2)=H(c_1)$, then multiplying by $w_1$ shows $wH(c_2)=w_1H(c_1)$, and we are done. If $w_2H(c_2)$ is parallel to $H(c_1)$ and separates $x$, $w_2x$, then multiplying by $w_1$ shows $wH(c_2)$ is parallel to $w_1H(c_1)$ and separates $w_1x$, $wx$. Since $w_1x$, $wx$ lie in the same halfspace of $w_1H(c_1)$, we get that $wH(c_2)$ must separate $wx$ and $x$. 

    Now suppose $c$ is not the terminal vertex of $P$. Let $P=P_1P_2$ be the decomposition such that $P_1$ terminates on $c$ and $P_2$ begins on $c$. Let $w=w_1w_2$ be the corresponding decomposition of labels. By the inductive hypothesis, $w_1H(c)$ separates $w_1x$ and $x$. By Definition~\ref{def:progressingAutomaton}(2), the points $x$, $w_2x$ lie in the same halfspace of $H(c)$. Multiplying by $w_1$, we get that $w_1x$, $wx$ lie in the same halfspace of $w_1H(c)$. Thus $w_1H(c)$ separates $x$ and $wx$. 
\end{proof}

Lemma~\ref{lem:checkpoints} justifies the names ``progressing automaton'' and ``checkpoint vertex''. Let $\Sigma$ be a progressing automaton, and let $x\in X^0$ be the basepoint. Let $P=e_1\dots e_n$ be an edge decomposition of a directed path beginning at $\s$, and let $w(1)\cdots w(n)$ be its label. Imagine tracing out $P$ edge-by-edge, considering subpaths $e_1\cdots e_i$ with $i$ iterated from $1$ to $n$. Simultaneously, consider the sequence of points $w(1)\cdots w(i)x$ in $X$. Lemma~\ref{lem:checkpoints} says the sequence of points $w(1)\cdots w(i)x$ never returns to $x$. But moreover, Lemma~\ref{lem:checkpoints} says each time a subpath $e_1\cdots e_i$ meets a checkpoint vertex $c\in C$, we get a hyperplane $w(1)\cdots w(i)H(c)$ which witnesses the separation of $w(1)\cdots w(i)x$ and $x$. This witness lasts until we arrive at the next ``checkpoint'', where we get a new witness of separation. The checkpoints mark the progression of the points $w(1)\cdots w(i)x$ away from $x$.

We now develop some tools useful for constructing progressing automata.

\begin{definition}[Progressing pair]
\label{def:progressingPair}
    Fix an action of $F_S$ on a $\mathrm{CAT}(0)$ cube complex $X$, and let $x\in X^0$ be a basepoint. Let $\s$ be a letter not in $S^\pm$. Let $\H$ be the set of visible hyperplanes, and let $w$ be a word over $S^\pm$.  

    \begin{itemize}
        \item A pair $(w,\s)$ is \emph{progressing} if $w'x\neq x$ for each $w'\in\pre(w)$, and $w H'$ separates $x$ and $w x$ for some $H'\in \H$.
        \item A pair $(w,H)$ with $H\in \H$ is \emph{progressing} if for each $w'\in\pre(w)$ the following two conditions hold:
        \begin{enumerate}
            \item $H$ does not separate $x$, $w' x$ for each $w'\in\pre(w)$.
            \item For some $H'\in \H$ either $w H'=H$, or $w H'$ is parallel to $H$ and separates $x$, $w x$.
        \end{enumerate}
    \end{itemize}
    
    In each of the above cases, we say $H'$ \emph{witnesses} the progression of $(w,\s)$ or $(w,H)$. 
\end{definition} 

\begin{definition}[Checkpoint tree]
\label{def:checkpointTree}
    Let $H\in \H\cup\{\s\}$. A \emph{checkpoint tree at $H$}, denoted $T_H$, is a finite rooted tree deterministically labeled over $S^\pm$ such that the following hold:
    
    \begin{enumerate}
        \item The root of $T_H$ is labeled $H$.
        \item The leaves of $T_H$ are labeled by elements of $\H$.
        \item If $P\to T_H$ is a rooted path labeled $w$ terminating on a leaf labeled $H'$, then $(w,H)$ is a progressing pair witnessed by $H'$.
    \end{enumerate}
\end{definition}

Checkpoint trees are our building blocks used to construct progressing automata. We informally describe the idea behind this construction. Imagine a checkpoint tree as a subgraph of a progressing automaton $T_H\subset \Sigma$ where $H\in \H$. As above, consider the process of walking along a path $P\to \Sigma$ edge-by-edge. When $P$ first meets the root of $T_H$, the hyperplane $w'H$ witnesses the separation of $w'x$, $x$, where $w'$ is the label of $P$ up to meeting the root of $T_H$. As $P$ travels along $T_H$, the hyperplane $w'H$ continues to be the witness of separation until $P$ reaches a leaf of $T_H$. Once $P$ reaches a leaf, a translate of $H'\in\H$ will be the new witness, where $H'$ is the label of the leaf.


\begin{definition}[$\lambda$-large tree]
    By definition, checkpoint trees $T_\s$ (i.e. with root labeled $\s$) have \emph{$\lambda$-large growth} for all $\lambda\in [0,1]$. For $H\in \H$, a checkpoint tree $T_H$ has \emph{$\lambda$-large growth} every non-leaf vertex has $\lambda$-large growth. 
\end{definition} 

The definition of $\lambda$-large growth for checkpoint trees is designed so that the automaton we construct from the trees will have $\lambda$-large growth. We now define that construction.

\begin{definition}[Realized automaton]
\label{def:realizedAutomaton}
    The \emph{realized automaton} of a set of checkpoint trees $\{T_H\}_{H\in\H\cup\{\s\}}$ is the automaton whose underlying graph is the union of the $T_H$ with vertices of the same label identified. The vertex $\s$ is the start vertex. 
\end{definition}

\begin{remark}
\label{rem:realizedAutomataDeterministic}
    The only vertices which are identified in Definition~\ref{def:realizedAutomaton} are the leaves and roots of the checkpoint trees $T_H$ with $H\in \H\cup \{\s\}$. Each root has a unique label in $\mathcal{H}\cup \{\s\}$, so no two roots are identified. Thus automata realized from checkpoint trees are deterministically labeled. 
\end{remark}

The following lemma is our main tool used to construct progressing automata. 

\begin{lemma}
\label{lem:realization}
    Suppose $\{T_H\}_{H\in\H\cup\{\s\}}$ is a collection of checkpoint trees with $\lambda$-large growth. Then the realized automaton $\Sigma$ is progressing and has $\lambda$-large growth.
\end{lemma}

\begin{proof}
    We view the checkpoint trees as subgraphs of $\Sigma$. Note that as subgraphs they may not be trees, since a leaf can be identified to the root of the same tree. Let the checkpoint vertices $C$ of $\Sigma$ be the roots and leaves of the $T_H$ for $H\in \H$. The labeling $H:C\to \H$ is induced by the labelings of the $T_H$.

    We first check Definition~\ref{def:progressingAutomaton}(1). Let $P\to \Sigma$ be a directed path beginning at $\s$ with label $w$ such that either $P$ does not meet $C$, or $P$ meets $C$ only at its terminal vertex. In both cases, $P$ is contained in $T_\s$. Let $P'\to T_\s$ be a (potentially trivial) extension of $P$ terminating on a leaf of $T_\s$. Let $H'$ be the label of this leaf, and let $w'$ be the label of $P'$. Note that $w\in\pre(w')$. By the definition of checkpoint tree, the pair $(w',\s)$ is progressing. Thus $w''x\neq x$ for each $w''\in \pre(w')$, and in particular $w''x\neq x$ for each $w''\in \pre(w)$, as desired. Additionally, by Definition~\ref{def:checkpointTree}(3), $H'$ witnesses the progression of $(w',\s)$. That is, $w'H'$ separates $x$ and $w'x$. If $P$ terminates on a checkpoint vertex, then $P=P'$ and $w=w'$, so we satisfy Definition~\ref{def:progressingAutomaton}(1).

    We now check Definition~\ref{def:progressingAutomaton}(2). Let $P\to \Sigma$ be a directed path with label $w$ such that either $P$ meets $C$ only at its initial vertex, or $P$ meets $C$ only at its initial and terminal vertices. In both cases, $P$ is contained in a tree $T_H$ for some $H\in \H$. Similarly to above, let $P'\to T_H$ be a (potentially trivial) extension of $P$ terminating on a leaf of $T_H$. Let $H'$ be the label of this leaf, and let $w'$ be the label of $P'$. Since $T_H$ is a checkpoint tree, the pair $(w',H)$ is progressing. Thus $x$ and $w''x$ lie in the same halfspace of $H$ for each $w''\in\pre(w')$. In particular, this also holds for $w''\in \pre(w)$. This shows Definition~\ref{def:progressingAutomaton}(2) holds if $P$ meets $C$ only at its initial vertex. By Definition~\ref{def:checkpointTree}(3), $H'$ witnesses the progression of $(w,H)$, so either $wH'=H$, or $w'H'$ is parallel to $H$ and separates $x$, $w'x$. When $P$ meets $C$ at its initial and terminal vertices, then $P=P'$ and $w=w'$, so these last separations and equalities show Definition~\ref{def:progressingAutomaton}(2) also holds in this case.

    We now show $\Sigma$ has $\lambda$-large growth. Let $K$ be the depth of $T_\s$. No vertex of a checkpoint tree is labeled $\s$ apart from the root of $T_\s$. Thus if a path $P\to \Sigma$ leaves $T_\s$, then its terminal vertex must lie in $T_H$ for some $H\in \H$. In particular, any directed path $P\to \Sigma$ beginning at $\s$ and of length at least $K$ terminates in $T_H$ for some $H\in \H$. If $P$ terminates on a non-leaf vertex of $T_H$, then that vertex has $\lambda$-large growth since $T_H$ has $\lambda$-large growth. Suppose $P$ terminates on a leaf of $T_H$, and let $H'$ be the label of the leaf. The leaf is identified with the root of $T_{H'}$ in $\Sigma$ which has $\lambda$-large growth since $T_{H'}$ has $\lambda$-large growth. Thus any $P\to \Sigma$ beginning at $\s$ and of length at least $K$ must terminate on a vertex with $\lambda$-large growth.
\end{proof}

One of our goals is to construct automata $\Sigma$ so that $wx\neq x$ for all $w\in \L_\Sigma$. We end this section by dealing with an easy case not requiring the machinery of progressing automata or checkpoint trees.

\begin{definition}[Fix-set $\Fix(x)$]
    Fix an action of $F_S$ on $X$, and let $x\in X^0$ be a basepoint. The \emph{fix-set} $\Fix(x)\subset S^\pm$ are those $s$ such that $sx=x$.  
\end{definition}

The notion of a progressing automaton will be most useful when $\Fix(x)$ is a small subset of $S^\pm$. If $\Fix(x)$ is large, then the following lemma allows us to construct automata with good growth whose accepted language acts non-trivially on $x$. 

\begin{lemma}
\label{lem:generatorsWithFixedPoint}
    Fix an action of $F_S$ on a $\mathrm{CAT}(0)$ cube complex $X$ without a global fixed-point, and let $x\in X^0$ be a basepoint. If $|\Fix(x)|\geq \lambda|S^\pm|$, then there exists an automaton $\Sigma$ with two vertices and $\lambda$-large growth such that $ wx\neq x$ for each $w\in \L_\Sigma$.
\end{lemma}

\begin{proof}
    By assumption, there exists some $s\in S^\pm$ such that $sx\neq x$. Let $\{\s,v\}$ be the vertex set of $\Sigma$. Add two directed edges from $\s$ to $v$ with labeled $s$ and $s^{-1}$. For each $s'\in\Fix(x)$ add a directed loop at $v$ labeled $s'$. Any directed path $P\to \Sigma$ beginning at $\s$ and of length at least $1$ terminates on $v$, which has $|\Fix(x)|\geq \lambda|S^\pm|$ outgoing edges. Thus $\Sigma$ has $\lambda$-large growth. By construction, any word $w\in \L_\Sigma$ is of the form $w=sw'$ where $s$ is as above and $w'$ is a word over $\Fix(x)$. Thus we have $ w x= s{w}'x= sx\neq x$.
\end{proof}

\section{Partitioning $\H$ and Progressing}

In this section, we concern ourselves with constructing checkpoint trees $T_H$ for each $H\in \H$. This requires a careful study of the arrangement of visible hyperplanes and their translates by $F_S$, along with the translates of a basepoint $0$-cube.   

\begin{definition}[Forward $\A$, backward $\B$, parallel $\P$]
\label{def:ABP}
    Fix an action of $F_S$ on a $\mathrm{CAT}(0)$ cube complex, and let $x\in X^0$ be a basepoint. Let $w$ be a word over $S^\pm$ and $H\in \H$ a visible hyperplane such that $x$ and $w x$ are both points of $N(H)$ in the same halfspace of $H$. We define a partition of $S^\pm$ relative to $H$ and $w$ as follows:

    \begin{itemize}
        \item The \emph{forward} elements $\A_w(H)$ are those $s\in S^\pm$ such that $wx$ and $ws x$ lie in the same halfspace of $H$ and $\dist(ws x,N(H))>0$.
        \item The \emph{backward} elements $\B_w(H)$ are those $s\in S^\pm$ such that $w x$ and $ws x$ lie in different halfspaces of $H$. 
        \item The \emph{parallel} elements $\P_w(H)$ are those $s\in S^\pm$ such that $w x$ and $ws x$ are both points of $N(H)$ in the same halfspace of $H$.
    \end{itemize}

    See Figure~\ref{fig:forwardBackward} for depictions of how elements in $\A_w(H)$ and $\B_w(H)$ behave. The elements of $\P_w(H)$ further partition into three sets:
    
    \begin{itemize}
        \item The \emph{visibly parallel} elements $\Psame_w(H)$ are those $s\in\P_w(H)$ such that $(ws)^{-1}H\in\H$.
        \item The \emph{crossing parallel} elements $\Pcross_w(H)$ are the non-visibly parallel $s\in\P_w(H)$ such that $ws H\pitchfork H$. 
        \item The \emph{disjoint parallel} elements $\Ppar_w(H)$ are the non-visibly parallel $s\in \P_w(H)$ such that $ws H\parallel~H$.
    \end{itemize}

    See Figure~\ref{fig:parallel} for depictions of how elements in $\Psame_w(H)$, $\Pcross_w(H)$, and $\Ppar_w(H)$ behave.
\end{definition}    

\begin{figure}[h]
    \centering
    \begin{subfigure}[h]{0.3\textwidth}
        \centering
        \includegraphics[width=\textwidth]{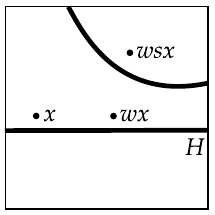}
        \caption{$s$ in $\A_w(H)$.}	
        \label{}
    \end{subfigure}
	\ \ \ \ \
    \begin{subfigure}[h]{0.3\textwidth}
        \centering	
	\includegraphics[width=\textwidth]{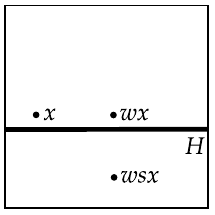}
        \caption{$s$ in $\B_w(H)$.}	
        \label{}
    \end{subfigure}
    \caption{Translates of $x$ by $\A_w(H)$ and $\B_w(H)$.}
    \label{fig:forwardBackward}
\end{figure}
    
When $w$ is the empty word, we will write $\A(H)$, $\B(H)$, and $\P(H)$ in place of $\A_w(H)$, $\B_w(H)$, and $\P_w(H)$ for simplicity. Similarly, $\Psame_w(H)$, $\Pcross_w(H)$, and $\Ppar_w(H)$ will be written $\Psame(H)$, $\Pcross(H)$, and $\Ppar(H)$ when $w$ empty.

\begin{remark}
\label{rem:visibleHyperplanes}
    We give a useful equivalent definition of visibly parallel elements. By definition, the visibly parallel elements $\Psame_w(H)$ are those elements $s\in \P_w(H)$ such that $(ws)^{-1}H\in\H$. Equivalently, $\Psame_w(H)$ are those $s\in \P_w(H)$ such that $\B_{ws}(H)$ is non-empty. Indeed, supposing $s\in \P_w(H)$ we have the following equivalences: $(ws)^{-1}H\in\H$ $\iff$ there exists $t\in S^\pm$ such that $(ws)^{-1}H$ is closest to $x$ separating $x$ and $tx$ (Figure~\ref{fig:visiblyParallel}) $\iff$ there exists $t\in S^\pm$ such that $H$ is closest to $ws x$ separating $ws x$ and $wst x$ $\iff$ there exists $t\in \B_{ws}(H)$. Since elements of $\Pcross_w(H)$ and $\Ppar_w(H)$ are non-visibly parallel by definition, we have that $\B_{ws}(H)$ is empty for any $s$ in $\Pcross(H)$ or $\Ppar(H)$.
\end{remark}

\begin{figure}[h]
    \centering
    \begin{subfigure}[h]{0.29\textwidth}
        \centering
        \includegraphics[width=\textwidth]{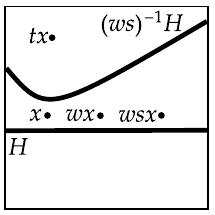}
        \caption{$s$ in $\Psame_w(H)$.}	
        \label{fig:visiblyParallel}
    \end{subfigure}
	\ \ \ \ \ 
    \begin{subfigure}[h]{0.29\textwidth}
        \centering
        \includegraphics[width=\textwidth]{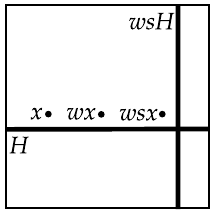}
        \caption{$s$ in $\Pcross_w(H)$.}	
        \label{}
    \end{subfigure}
	\ \ \ \ \ 
    \begin{subfigure}[h]{0.29\textwidth}
        \centering	
	\includegraphics[width=\textwidth]{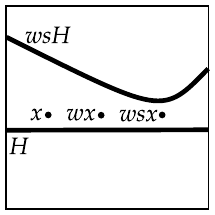}
        \caption{$s$ in $\Ppar_W(H)$.}	
        \label{}
    \end{subfigure}
    \caption{Translate of $x$ by $\Psame_w(H)$, $\Pcross_w(H)$, and $\Ppar_w(H)$.}
    \label{fig:parallel}
\end{figure}

\subsection{Easy progressions}\hfill

\begin{lemma}
\label{lem:easyProgression}
    Fix an action of $F_S$ on a $\mathrm{CAT}(0)$ cube complex $X$, and let $x\in X^0$ be a basepoint. Suppose for some $H\in \H$ and a word $w$ that $x,wx\in N(H)$ lie in the same halfspace of $H$. The following hold:
    
    \begin{enumerate}
        \item For each $s\in \A_w(H)$ the pair $(ws,H)$ is progressing.
        \item For $s\in \Psame(H)$, the pair $(ws,H)$ is progressing.  
        \item For each $s\in \Ppar_w(H)$ and $t\in \B(H)$, the pair $(wst,H)$ is progressing. 
    \end{enumerate}
\end{lemma}

\begin{proof}
    $(1)$ By definition of $\A_w(H)$, $\dist(ws x, N(H))>0$. Thus there exists a hyperplane separating $ws x$ and $N(H)$ by Lemma~\ref{lem:hyperplanes}. Any hyperplane separating $wsx$, $N(H)$ is parallel to $H$ and separates $wsx$, $wx$ since $wx\in N(H)$. Let $H''$ be such a hyperplane closest to $wsx$. Multiplying by $(ws)^{-1}$, we get that $H'=(ws)^{-1}H''$ is a closest hyperplane to $x$ separating $x$, $s^{-1}x$. Thus $H'\in \H$, and $H'$ witnesses the progression of $(ws,H)$. 

    $(2)$ By definition of $\Psame_w(H)$, the hyperplane $H$ does not separate $w x$, $w s x$. Also by definition, $(ws)^{-1}H$ lies in $\H$. Setting $H'=(ws)^{-1}H$, we have $wsH'=H$. Thus $H'$ witnesses the progression of $(ws,H)$.

    $(3)$ By the definition of $\Ppar_w(H)$, the hyperplane $H$ does not separate two of $x,wx,wsx\in N(H)$. Multiplying by $ws$ and using the definition of $\Ppar_w(H)$, we see $wsH$ is parallel to $H$ and lies in the same halfspace of $H$ as $x,wx,wsx$ because $wsx\in N(H)\cap N(ws H)$. By definition of $\B(H)$, $H$ separates $x$ and $tx$. Thus $wsH$ separates $wstx$ and $N(H)$. In particular, there exists a hyperplane separating $wstx$ and $N(H)$. Let $H''$ be such a hyperplane closest to $wstx$. Note that $H''$ is parallel to $H$ and separates $wstx$ and $wsx$, since $wsx\in N(H)$. Furthermore, multiplying by $(wst)^{-1}$ shows $H'=(wst)^{-1}H''$ is a closest hyperplane to $x$ separating $x$, $t^{-1}x$. Thus $H'\in \H$, and $H'$ witnesses the progression of $(ws,H)$.
\end{proof}

\begin{definition}[Tilde notation, child function $c(v)$]
    Suppose $T$ is a rooted tree deterministically labeled over $S^\pm$. We define an operation interchanging between vertices of $T$ and words over $S^\pm$: For a vertex $v\in T$, $\tilde v$ is the label of the unique rooted, directed path $P\to T$ terminating on $v$. For a word $w$ over $S^\pm$, $\tilde w$ denotes the terminal vertex of the rooted, directed path $P\to T$ with label $w$. If no such path exists, then $\tilde w$ is undefined. For a vertex $v$ of $T$, $c(v)$ denotes the set of children of $v$, and $\tilde c(v)$ denotes the set of labels on outgoing edges at $v$.  
\end{definition}

Note that $\tilde v$ is uniquely determined for $v\in T$ since $T$ is deterministically labeled, and $\tilde c(v)$ is exactly the set of labels on edges joining $v$ to a vertex in $c(v)$. 

\begin{corollary}
\label{cor:easyProgression}
    Let $T$ be a finite tree deterministically labeled over $S^\pm$ with a root labeled $H\in \H$ such that each non-leaf vertex has $\lambda$-large growth. Suppose for each leaf $v\in T$ that one of the following holds:
    
    \begin{enumerate}
        \item $|\A_{\tilde v}(H)|\geq\lambda|S^\pm|$
        \item $|\Psame_{\tilde v}(H)|\geq\lambda|S^\pm|$
        \item $|\Ppar_{\tilde v}(H)|\geq\lambda|S^\pm|$ and $|\B(H)|\geq\lambda|S^\pm|$
    \end{enumerate}
    
    Then $T$ can be extended to a checkpoint tree $T_H$ with $\lambda$-large growth. 
\end{corollary}

The extension $T_H$ of $T$ is formed by attaching outgoing edges to leaves either once or twice and then labeling the new leaves over $\H$. Note that Corollary~\ref{cor:easyProgression} also applies when $T$ is a single vertex $v$ labeled by $H$. In this case, $\tilde v$ is the empty word.

\begin{proof}
    Let $v$ be a leaf of $T$, and add labeled outgoing edges at $v$ in the following three ways.
    
    For each $s\in \A_{\tilde v}(H)$ add an edge outgoing from $v$ labeled $s$. A rooted path terminating at a leaf of an added edge has label $\tilde vs$ with $s\in \A_{\tilde v}(H)$. By Lemma~\ref{lem:easyProgression}(1), $(\tilde vs, H)$ is progressing. Thus we can label the terminal vertex with some $H'\in\H$ witnessing the progression of $(\tilde vs, H)$. 
    
    We perform a similar construction for $\Psame_w(H)$. Add an outgoing edge at $v$ with label $s$ for each $s\in \Psame_{\tilde v}(H)$. A rooted path terminating at a leaf of an added edge has label $\tilde vs$ with $s\in \Psame_{\tilde v}(H)$. Lemma~\ref{lem:easyProgression}(2) guarantees that each $(\tilde vs, H)$ is progressing, so we can again label the the terminal vertex of the path with some $H'\in \H$ witnessing the progression.
    
    If $|\B_{\tilde v}(H)|\geq\lambda|S^\pm|$, then we add the following edges, otherwise we add no edges of this form. First, for each $s\in \Ppar_{\tilde v}(H)$ add an outgoing edge at $v$ labeled $s$. Second, at each of the newly created leaves, add an outgoing edge labeled $t$ for each $t\in \B(H)$. In this way we have created new leaves which are grandchildren of $v$, i.e., children of children. A rooted path terminating at a grandchild leaf has label $\tilde v st$ with $s\in \Ppar_{\tilde v}(H)$ and $t\in \B(H)$. By Lemma~\ref{lem:easyProgression}(3), the pair $(\tilde vst, H)$ is progressing. We label this new leaf with some $H'\in \H$ witnessing this progression. 

    The tree $T_H$ is formed by attaching edges at each leaf of $T$ as above. By construction, the tree $T_H$ is progressing. By the hypothesis on leaves of $T$, each leaf $v\in T$ receives at least $\lambda|S^\pm|$ outgoing edges. In the third type of edge addition, since we require $|\B_{\tilde v}(H)|\geq\lambda|S^\pm|$, each child of $v$ corresponding to an element of $\Ppar_{\tilde v}(H)$ receives at least $\lambda|S^\pm|$ outgoing edges. Thus $T_H$ has $\lambda$-large growth.
\end{proof}

\subsection{Progressing with \texorpdfstring{$\P$}{parallels} and \texorpdfstring{$\B$}{backwards}}\hfill

The following property along with the finite-dimensionality of the $\mathrm{CAT}(0)$ cube complex $X$ will help to bound the size of checkpoint trees constructed in this subsection.

\begin{definition}[Rooted lifts, property $(\star)$]
    Let $T$ be a labeled, rooted tree. We say that a rooted, directed path $P\to T$ is a \emph{rooted lift} of a directed path $P'\to T$ if $P$ and $P'$ have the same label $w$. $T$ has \emph{property $(\star)$} if every directed path has a rooted lift.
\end{definition}



In this subsection, we inductively construct checkpoint trees by repeatedly adding new leaves. The following lemma will enable us to maintain property $(\star)$ at each step of the induction.

\begin{lemma}
\label{lem:inductiveStar}
    Let $T$ be a deterministically labeled, rooted tree. Let $T_0\subset T$ be the subtree spanned on all non-leaf vertices. Assume $T_0$ has property $(\star)$. Then $T$ has property $(\star)$ if and only if for every leaf $v\in T_0$ we have $\tilde c(v)\subset \bigcap_{u\in \suf(\tilde v)}\tilde c(\tilde u)$. 
\end{lemma}

\begin{figure}[h]
    \centering
    \includegraphics[width=0.5\textwidth]{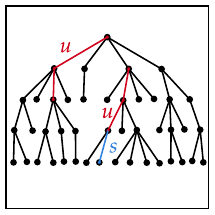}
    \caption{The path $P'$ is labeled $us$. The red subpath is $P'_0$, and the blue edge is $e$. The rooted red path is the lift $P_0$ of $P'_0$. There is a rooted lift of $P$ if and only if there is an edge labeled $s$ at the terminal end of $P_0$.}
    \label{fig:tree}
\end{figure}

\begin{proof}
    Let $P'\to T$ be a directed path. If the image of $P'$ is contained in $T_0$, then by assumption, there exists a rooted lift $P\to T_0$. Since $T_0\subset T$, the lift $P\to T_0$ is also a lift $P\to T$. Thus $T$ has property $(\star)$ if and only if partial rooted lifts of paths terminating on a leaf extend to rooted lifts.
    
    Any directed path $P'\to T$ terminating on a leaf of $T$ has the form $P'=P'_0e$, where $P'_0$ terminates on a leaf $v$ of $T_0$ and $e$ is an edge. Moreover, the label of $P'$ has the form $us$, where $s\in \tilde c(v)$ is the letter labeling $e$ and $u\in \suf(\tilde v)$ is the word labeling $P'_0$. Since $T_0$ has property $(\star)$, there exists a rooted lift $P_0\to T$ of $P'_0\to T$. The rooted lift $P_0\to T$ terminates on the vertex $\tilde u$, and the lift extends to a lift of $P'\to T$ if and only if there exists an outgoing edge at $\tilde u$ labeled $s$. See Figure~\ref{fig:tree}. This last condition can be equivalently written $s\in \tilde c(\tilde u)$. 
    
    We now find conditions which are equivalent to various paths having rooted lifts. We consider larger and larger collections of paths. Ultimately we show the condition in the lemma is equivalent to having rooted lifts of all paths terminating on leaves. As discussed above, this is equivalent to $T$ having property $(\star)$. 
    
    In the decomposition $P'=P'_0e$, letting $e$ vary over the outgoing edges of~$v$, we see that $P'$ has a rooted lift if and only if $\tilde c(v)\subset \tilde c(\tilde u)$. If we let $P'_0$ vary among paths $P'_0\to T$ terminating on $v$, then we see that $P'$ has a rooted lift if and only if $\tilde c(v)\subset\tilde c(\tilde u)$ for each $u\in \suf(\tilde v)$. In other words, $\tilde c(v)\subset\bigcap_{u\in \suf(\tilde v)}\tilde c(\tilde u)$. Finally, letting $P'=P'_0e$ vary over all directed paths $P'\to T$ terminating on a leaf, we see that $P'\to T$ has a rooted lift if and only if $\tilde c(\tilde u)\subset \bigcap_{u\in \suf(\tilde v)}\tilde c( t)$ for every leaf $v\in T_0$. 
\end{proof}

\begin{lemma}
\label{lem:keyLemma}
    Fix an action of $F_S$ on a $\mathrm{CAT}(0)$ cube complex $X$, and fix a basepoint $x\in X^0$. There exists a decreasing function $\epsilon(n)>0$ and a function $\alpha(\epsilon_0,\epsilon_1,n)>0$ defined for $\epsilon_0\in(0,\epsilon(n)]$, $\epsilon_1\in[\epsilon_0,1)$ such that the following holds: If no more than $n$ hyperplanes in $\H$ pairwise cross and $\epsilon_0|S^\pm|\leq|\B(H)|\leq\epsilon_1|S^\pm|$ for some $H\in \H$, then there exists a progressing checkpoint tree at $H$ with $\alpha(\epsilon_0,\epsilon_1,n)$-large growth.
\end{lemma}

For $\epsilon_0>0$, inductively define a diagonal array of real numbers $D_i(j)$, $i\in \mathbf N$, $0\leq j< i$, as follows. For the base case, define $D_1(0)=1-3\epsilon_0$. For $i+j>1$, define $D_i(j)$ by the following two formulas:

    \begin{itemize}
        \item For $j<i-1$, $D_i(j)=(1-\epsilon_0)D_{i-1}(j)-\sum_{0\leq k<j}(1-D_i(k))$.
        \item For $j=i-1$, $D_i(j)=1-3\epsilon_0-\sum_{0\leq k<j} (1-D_i(k))$.
    \end{itemize}

\begin{remark}
\label{rem:dij}
    Each $D_i(j)$ is defined only in terms of those $D_{i'}(j')$ with $i'+j'<i+j$, and $D_i(j)\leq D_{i'}(j')$ for $i'+j'<i+j$. 
\end{remark}

\begin{lemma}
\label{lem:limit}
    For any $i\in \mathbf N$ and $0\leq j\leq i$, we have $D_i(j)\to 1$ as $\epsilon_0\to 0$.
\end{lemma}

\begin{proof}
    We prove the lemma by induction. $D_1(0)=1-3\epsilon_0$ by the second equation, so $D_1(0)\to 1$ as $\epsilon_0\to 0$. Suppose the lemma holds for all $D_{i'}(j')$ with $i'+j'<i+j$. From the defining equations of $D_i(j)$ it can be seen that $D_i(j)\to 1$ as $\epsilon_0\to 0$ if each $D_{i'}(j')\to 1$ as $\epsilon_0\to 0$. 
\end{proof}

\begin{proof}[Proof of Lemma~\ref{lem:keyLemma}]
    We will only consider the partition of $\H$ defined relative to $H$, so we use $\A_w$, $\B_w$, and $\P_w$ in place of $\A_w(H)$, $\B_w(H)$, and $\P_w(H)$ for the duration of this proof. 
    
    Let $F=S^\pm-\B$, and note $|F|\geq (1-\epsilon_1)|S^\pm|$ since $|\B|\leq \epsilon_1|S^\pm|$. If $|\A|\geq\epsilon_0|F|$, $|\Psame|\geq\epsilon_0|F|$, or $|\Ppar|\geq\epsilon_0|F|$, then by Corollary~\ref{cor:easyProgression} there is a progressing tree at $H$ with $\epsilon_0(1-\epsilon_1)$-large growth. Otherwise, we have $|\Pcross|\geq (1-3\epsilon_0)|F|$, since $\Pcross\subset F$ and $F$ is partitioned between $\A$, $\Psame$, $\Ppar$, and $\Pcross$. Supposing $|\Pcross|\geq (1-3\epsilon_0)|F|$, we will construct a progressing checkpoint tree $T$ with $\epsilon_0(1-\epsilon_1)D_n(n-1)$-large growth. Supposing we have constructed such a tree, the lemma easily follows. Indeed, the value of $\epsilon_0(1-\epsilon_1)D_n(n-1)$ is positive for $\epsilon_1<1$ and sufficiently small $\epsilon_0>0$ by Lemma~\ref{lem:limit}. Let $\alpha(\epsilon_0,\epsilon_1,n)\defeq \epsilon_0(1-\epsilon_1)D_n(n-1)$ and define $\epsilon(n)$ to be a constant such that $\alpha(\epsilon_0,\epsilon_1,n)$ is positive for all $\epsilon_0\leq \epsilon(n)$. Additionally, choose $\epsilon(n)$ be be decreasing in $n$.
    
    
    We construct $T$ as the limit of a finite sequence of deterministically labeled trees $T_1,T_2,\ldots, T_m=T$. For $i<m$, all leaves of $T_i$ have depth $i$. Given $T_i$, the next tree $T_{i+1}$ is constructed by deleting some edges from $T_i$ and then adding children to all remaining leaves. Furthermore, for $i<m$, each $T_i$ will satisfy the following:

    \begin{enumerate}
        \item $T_i$ has property $(\star)$.
        \item For any vertex $v\in T_i$, $\tilde c(v)\subset \Pcross_{\tilde v}\cap F$.
        \item A vertex of depth $j<i$ in $T_i$ has at least $D_i(j)|F|$ children.
    \end{enumerate}
    
    
    
    

    Let $T_1$ be the rooted, depth-one tree with an edge for each element of $\Pcross$, and label each edge by its corresponding element in $\Pcross$. $T_1$ trivially satisfies condition (1) and satisfies condition (2) by definition. Condition (3) holds by the assumption $|\Pcross|\geq (1-3\epsilon_0)|F|$.

    Suppose we have constructed $T_1,\ldots, T_i$. Let $v$ be a leaf of $T_i$, and let $u$ be its ancestor. By condition (2), the label $s$ of the edge joining $u$ to $v$ lies in $\Pcross_{\tilde u}\cap F$. By Remark~\ref{rem:visibleHyperplanes} and since $\tilde us=\tilde v$, it follows that $\B_{\tilde v}$ is empty. Thus the set $F$ is partitioned between $\A_{\tilde v}$, $\Psame_{\tilde v}$, $\Pcross_{\tilde v}$, and $\Ppar_{\tilde v}$. There are two cases: (i) at least one of $|\A_{\tilde v}|>\epsilon_0|F|$ or $|\Psame_{\tilde v}|>\epsilon_0|F|$ or $|\Ppar_{\tilde v}|>\epsilon_0|F|$; otherwise (ii) $|\Pcross_{\tilde v}\cap F|\geq (1-3\epsilon_0)|F|$. We define a partition of the vertices of $T_i$ into sets $\RN{1}$ and $\RN{2}$. The partition is defined inductively from larger depth vertices to lower depth vertices, beginning with the leaves of $T_i$. 
    
    \begin{itemize}
        \item If $v$ is a leaf of $T_i$ such that (i) holds, then $v\in \RN{1}$. Otherwise (ii) holds and $v\in \RN{2}$.
        \item Let $u$ be a vertex whose children $c(u)$ have been partitioned between $\RN{1}$ and $\RN{2}$. If $|c(u)\cap \RN{1}|\geq \epsilon_0|c(u)|$, then $u\in \RN{1}$. Otherwise $|c(u)\cap \RN{2}|>(1-\epsilon_0)|c(u)|$ and $u\in \RN{2}$.
    \end{itemize}
    
    We construct $T_{i+1}$ differently depending on whether the root of $T_i$ is in $\RN{1}$ or $\RN{2}$. If the root is in $\RN{1}$, then $T_{i+1}$ will be the final tree in the sequence, i.e., $m=i+1$. Supposing the root is in $\RN{1}$, begin by deleting all vertices in $\RN{2}$ from $T_i$. Note that $T_i-\RN{2}$ may be disconnected. Let $T'_i$ be the component of $T_i-\RN{2}$ containing the root. By condition (3) and construction of $T'_i$, a depth $j<i$ vertex of $T'_i$ has at least $\epsilon_0D_i(j)|F|$ children. By Remark~\ref{rem:dij} and the inequality $|F|\geq (1-\epsilon_1)|S^\pm|$, we have $\epsilon_0D_i(j)|F|\geq \epsilon_0(1-\epsilon_1)D_n(n-1)|S^\pm|$. Thus depth $j<i$ vertices of $T'_i$ have $\epsilon_0(1-\epsilon_1)D_n(n-1)$-growth, which is our desired growth. For each leaf $v\in T'_i$ one of $|\A_{\tilde v}|>\epsilon_0|F|$ or $|\Psame_{\tilde v}|>\epsilon_0|F|$ or $|\Ppar_{\tilde v}|>\epsilon_0|F|$ holds. By the inequality $|F|\geq (1-\epsilon_1)|S^\pm|$ and Corollary~\ref{cor:easyProgression}, we can extend $T_{i}$ to a progressing tree $T_{i+1}$ with $\epsilon_0(1-\epsilon_1)D_n(n-1)$-large growth.

    Suppose the root of $T_i$ is in $\RN{2}$. Let $T'_i$ be the component of $T_i-\RN{1}$ containing the root, and let $\mathcal T$ be the maximal subtree of $T'_i$ satisfying property $(\star)$. There is a filtration $\mathcal T_1\subset \cdots \subset\mathcal T_i=\mathcal T$ where $\mathcal T_j$ is the maximal rooted, depth $j$ subtree of $T'_i$ satisyfing property $(\star)$. Note that unions of rooted subtrees of $T'_i$ satisfying $(\star)$ also satisfy $(\star)$, so the $\mathcal T_j$ are well-defined. Additionally, for each $j$, the subtree spanned on the depth at most $j-1$ vertices in $\mathcal T_{j}$ is exactly $\mathcal T_{j-1}$. 
    
    We now focus our attention on $T'_i$ and the filtration $\mathcal T_1\subset\cdots\subset \mathcal T_i$. Until otherwise mentioned, the tilde notation (i.e. $\tilde v$) and child notation (i.e. $c(v)$) are understood to be relative to $T'_i$. For $v\in \mathcal T_j$, we let $\tilde c_j(v)$ denote labels of edges outgoing from $v$ and contained in $\mathcal T_j$. Let $v$ be a depth $j$ vertex in $\mathcal T_{j+1}\subset T'_i$ with $j+1<i$. Note $T_j$ is the subtree of $T_{j+1}$ spanned on all non-leaf vertices. By Lemma~\ref{lem:inductiveStar} and maximality of $\mathcal T_{j+1}$, the set of labels on outgoing edges at $v$ is exactly $\tilde c(v)\cap \big(\bigcap_{u\in \suf(\tilde v)}\tilde c_j(\tilde u)\big)$.
    
    We will inductively show that a depth $j$ vertex $v\in \mathcal T_{j+1}$ has $D_{i+1}(j)|F|$ children. Since the trees are deterministically labeled, this is equivalent to showing $|\tilde c_{j+1}(v)|\geq D_{i+1}(j)|F|$. As a base case, the root of $\mathcal T_1$ has $D_{i+1}(0)=(1-\epsilon_0)D_i(0)|F|$ children since the root was in $\RN{2}$ and the root of $T_i$ had $D_i(0)|F|$-many children by condition (3). For higher depth vertices, we have the following bound:

    \begin{align*}
        \bigg|\tilde c(v)\cap \big(\bigcap_{u\in \suf(\tilde v)-\{\tilde v\}}\tilde c_j(\tilde u)\big) \bigg| &\geq |\tilde c(v)| -\sum_{u\in \suf(\tilde v)-\{\tilde v\}}|F- \tilde c_j(\tilde u)|\\ 
        &= (1-\epsilon_0)D_i(j)|F| - \sum_{0\leq k< j} (1-D_{i+1}(k))|F|\\ 
        &= D_{i+1}(j)|F|.
    \end{align*}
    
    In particular, we have established that a depth $j<i$ vertex of $\mathcal T_i=\mathcal T$ has $D_{i+1}(j)|F|$ children. 
    
    We now examine the tree $\mathcal T$. The tilde notation (e.g. $\tilde v$) and child notation (e.g. $c(v)$) are now understood to be relative to $\mathcal T$. To finish the construction of $T_{i+1}$, we add children to the leaves of $\mathcal T$. Let $v\in \mathcal T$ be a leaf, and recall the leaves all have depth $i$. By definition of $\RN{2}$, $|\Pcross_{\tilde v}\cap F|\geq (1-3\epsilon_0)|F|$, so we have the following bound:

    \begin{align*}
        \bigg|\Pcross_{\tilde v}\cap F\cap \big(\bigcap_{u\in \suf(\tilde v)-\{\tilde v\}}\tilde c(\tilde u)\big)\bigg| &\geq |\Pcross_{\tilde v}\cap F| - \sum_{u\in \suf(\tilde v)-\{\tilde v\}} |F-\tilde c(\tilde u)|\\
        &= (1-3\epsilon_0)|F| - \sum_{0\leq k<i} (1-D_{i+1}(k))|F|\\
        &=D_{i+1}(i)|F|
    \end{align*}

    For each element $s\in \Pcross_{\tilde v}\cap F\cap \big(\bigcap_{u\in \suf(\tilde v)-\{\tilde v\}}\tilde c(\tilde u))$, we add an edge with an endpoint at $v$ with label $s$. Thus $v$ has at least $D_{i+1}(i)|F|$ children. The tree $T_{i+1}$ is the result of adding children to each leaf $v\in \mathcal T$ as above. By Lemma~\ref{lem:inductiveStar}, $T_{i+1}$ has property $(\star)$ and thus satisfies condition (1). All new edges have labels from the relevant $\Pcross_{\tilde v}\cap F$ sets, so $T_{i+1}$ also satisfies condition (2). Finally, the inequalities above show that $T_{i+1}$ satisfies condition (3).

    We have finished defining the sequence $T_1, T_2,\ldots ,$ and we now show that the sequence terminates in a bounded number of steps. Conditions (1) and (2) have the following consequence: For any vertex $v\in T_i$, the set of hyperplanes $\{u H:u\in \pre(\tilde v)\}$ pairwise cross. We first show that ${\tilde v} H\pitchfork uH$ for all $u\in \pre(\tilde v)$. Fixing a prefix $u$ of $\tilde v$, we can write $\tilde v$ in the form $\tilde v=uws$, where $s$ is the final letter of $\tilde v$ and $w$ is a (potentially empty) subword. We have $s\in \Pcross_{w}$ by conditions (1) and (2). By definition of $\Pcross_{w}$, we have $ws H \pitchfork H$. Multiplying by $u$ and using $uws={\tilde v}$, we get ${\tilde v} H\pitchfork uH$ as desired. For any $u\in\pre(\tilde v)$, we can repeat the above argument with $\tilde u$ in the place of $v$ to conclude $u H\pitchfork w H$ for all $w\in\pre(u)$. This suffices to show the hyperplanes in $\{u H:u\in \pre(\tilde v)\}$ pairwise cross.  
    
    The above analysis allows us to bound the number of trees in our sequence. By assumption, no more than $n$ hyperplanes in $\H$ can pairwise cross. The argument above shows that associated to a depth $k$ vertex $v\in T_i$ is a set $\{u H:u\in \pre(\tilde v)\}$ of $k+1$ pairwise crossing hyperplanes. Thus the $T_i$ can have depth at most $n-1$, and our sequence must terminate in at most $n-1$ steps, i.e., $m\leq n-1$. The inequality $m\leq n-1$, condition (3), and Remark~\ref{rem:dij} imply that $T=T_m$ has at least $\epsilon_0(1-\epsilon_1)D_n(n-1)$-growth.  
\end{proof}

\subsection{An inductive construction}\hfill

\begin{definition}[Hyperplane inversion]
    A group $G$ acts on a $\mathrm{CAT}(0)$ cube complex $X$ \emph{without hyperplane inversions} if there does not exist $g\in G$ and a hyperplane $H$ of $X$ such that $gH=H$ and $g$ interchanges the halfspaces of $H$.
\end{definition}

Group actions can always be arranged to be without hyperplane inversions by passing to the first cubical subdivision. That is, for any action of $G$ on a $\mathrm{CAT}(0)$ cube complex $X$, there is an induced action of $G$ on the first cubical subdivision of $X$, and this latter action is always without hyperplane inversions. The induced action has a global fixed-point if and only if the original action has a global fixed-point. Thus we can freely assume that our group actions are without hyperplane inversions. 

\begin{figure}[h]
    \centering
    \begin{subfigure}[h]{0.299\textwidth}
        \centering
        \includegraphics[width=\textwidth]{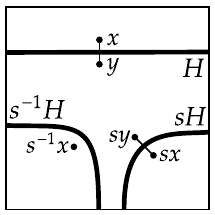}
        \caption{Both $s, s^{-1}$ in $\B(H)$.}	
        \label{fig:bound1}
    \end{subfigure}
	\ \ \ \ \ 
    \begin{subfigure}[h]{0.3\textwidth}
        \centering
        \includegraphics[width=\textwidth]{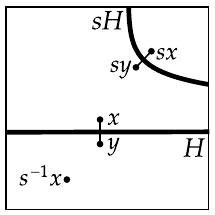}
        \caption{Only $s^{-1}$ in $\B(H)$.}	
        \label{fig:bound2}
    \end{subfigure}
    \caption{Comparing $\dist(x,sx)$ and $\dist(y,sy)$.}
    \label{}
\end{figure}

\begin{lemma}
\label{lem:backwardsBound}
    Fix an action of $F_S$ on a $\mathrm{CAT}(0)$ cube complex $X$ without hyperplane inversions, and let $x\in X^0$ be a basepoint that minimizes $\sum_{s\in S^\pm}\dist(x,sx)$. Then $\B(H)\leq \frac12 |S^\pm|$ for any $H\in \H$.
\end{lemma}

\begin{proof}
    Suppose $\B_H>\frac12 |S^\pm|$ for some $H\in \H$. Let $y$ be the $0$-cube opposite $x$ in $N(H)$. We show $\sum_{s\in S^\pm}\dist(y,sy)<\sum_{s\in S^\pm}\dist(x,sx)$, contradicting minimality. For each pair $s$, $s^{-1}$ we have the following three possibilities:

    \begin{enumerate}
    \item Both $s,s^{-1}\in \B(H)$. We show that $d(y,sy)=d(x,sx)-2$. By assumption $H$ is a closest hyperplane to $x$ separating $s^{-1}x$, $x$. Multiplying by $s$ shows that $sH$ is a closest hyperplane to $sx$ separating $x$, $sx$. We also have that $H$ is a closest hyperplane to $x$ separating $x$, $sx$ by assumption. The points $y$ and $sy$ are opposite to $x$ and $sx$ across the hyperplanes $H$ and $sH$. Thus, as long as $s^{-1}H$ and $H$ are distinct, then $y$, $sy$ are separated by two fewer hyperplanes than $x$, $sx$. See Figure~\ref{fig:bound1}. If $s^{-1}H=H$, then $s$ stabilizes $H$ and interchanges its halfspaces, contradicting no hyperplane inversions. Thus $s^{-1}H\neq H$ and $d(y,sy)=d(x,sx)-2$.
    
    \item Exactly one of $s\in \B(H)$ or $s^{-1}\in \B(H)$. We show that $\dist(y,sy)=\dist(x,sx)$. The cases are symmetric, so we deal just with the case $s^{-1}\in \B(H)$. By assumption $H$ is a closest hyperplane to $x$ separating $x$, $s^{-1}x$. Multiplying by $s$ we get that $sH$ is a closest hyperplane to $sx$ separating $sx$, $x$. By assumption the hyperplane $H$ does not separate $sx$, $x$. Since $sy$ and $y$ are opposite to $sx$ and $x$ across the hyperplanes $sH$ and $H$, the points $sy$, $y$ are separated by the same number of hyperplanes as $sx$, $x$. See Figure~\ref{fig:bound2}. Hence $\dist(y,sy)=\dist(x,sx)$.
    
    \item Both $s,s^{-1}\not\in \B(H)$. In this case we immediately have $\dist(y,sy)\leq \dist(x,sx)+2$. The triangle inequality implies $\dist(y,sy)\leq \dist(y,x)+\dist(x,sx)+\dist(sx,sy)$, and $\dist(x,y)=\dist(sx,sy)=1$ since $x$, $y$ are opposite.
\end{enumerate}

Let the partition $S^\pm=\RN{1}\sqcup \RN{2}\sqcup \RN{3}$ correspond to the three cases above, e.g., $s,s^{-1}\in \RN{2}$ if the pair $s$, $s^{-1}$ satisfies case (2). Note that $|\RN{1}|>|\RN{3}|$ since $|\B(H)|>\frac12 |S^\pm|$. Thus we have the following inequality.
\[
\sum_{s\in S^\pm}\dist(x,sx)\geq\sum_{s\in \RN{1}}\big( \dist(y,sy)+2 \big) +\sum_{s\in \RN{2}}\dist(y,sy) + \sum_{s\in \RN{3}}\big(\dist(y,sy)-2\big)>\sum_{s\in S^\pm}\dist(y,sy).
\]
\end{proof}

In our construction of progressing automata, we will sometimes pass to a subset of $S^\pm$ and construct an automaton over the subset. The following notation facilitates this. 

\begin{definition}[$\H_{\mathsf P}$ for $\mathsf P\subset S^\pm$]
    Let $\mathsf P$ be a subset of  $S^\pm$. Then $\H_{\mathsf P}\subset \H$ are those hyperplanes which separate $x$ and $sx$ for some $s\in \mathsf P$. Equivalently, $\H_{\mathsf P}= \bigcup_{s\in \mathsf P}\H_s$. 
\end{definition} 

\begin{lemma}
\label{lem:induction}
    Fix an action of $F_S$ on a $\mathrm{CAT}(0)$ cube complex $X$, and let $x\in X^0$ be a basepoint such that $\Fix(x)$ is empty. Suppose that no more than $n$ hyperplanes in $\H$ pairwise cross. If $\P(H)\geq \lambda|S^\pm|$ for some $H\in \H$ and $\lambda\in (0,1)$, then there exists a symmetric $\mathsf S^\pm\subseteq S^\pm$ such that $|\mathsf S^\pm|\geq (2\lambda-1)|S^\pm|$ and no more than $n-1$ hyperplanes of $\H_{\mathsf S^\pm}$ pairwise cross.
\end{lemma}

\begin{proof}
    Suppose $H\in \H$ satisfies $\P(H)\geq \lambda|S^\pm|$. Since $sx\neq x$ for $s\in \P(H)$, each hyperplane in $\H_{\P(H)}$ crosses $H$, so at most $n-1$ hyperplanes in $\H_{\P(H)}$ pairwise cross. See Figure~\ref{fig:induction}. The intersection $\P(H)\cap \P(H)^{-1}$ is clearly symmetric, and thus can be written as $\mathsf S^\pm=\mathsf S\sqcup \mathsf S^{-1}$ for some $\mathsf S\subset S$. We have the bound $|\mathsf S^\pm|= |\P(H)\cap \P(H)^{-1}| \geq \lambda|S^\pm|-(1-\lambda)|S^\pm|= (2\lambda-1)|S^\pm|$, and no more than $n-1$ hyperplanes in $\H_{\mathsf S^\pm}$ pairwise cross since $\H_{\mathsf S^\pm}\subset \H_{\P(H)}$. 
\end{proof}

\begin{figure}[h]
    \centering
    \includegraphics[width=0.3\textwidth]{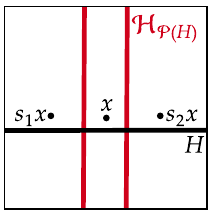}
    \caption{The hyperplanes in $\H_{\P(H)}$ are colored red. Each hyperplane in $\H_{\P(H)}$ crosses $H$.}
    \label{fig:induction}
\end{figure}

The above lemma will allow us to induct on the number of pairwise crossing hyperplanes in $\H$. To effectively do this, we need the following bound. 

\begin{lemma}
\label{lem:backwardsInductiveBound}
    Let $S$ and $\mathsf S$ be as in Lemma~\ref{lem:induction}. If $\B(H)\leq\beta|S^\pm|$ for each $H\in \H$, then $\B(\mathsf H)\leq \frac{\beta}{2\lambda-1}|\mathsf S^\pm|$ for each $\mathsf H\in \H_{\mathsf S^\pm}$.
\end{lemma}

\begin{proof}
    Combine $\B(H)\leq \beta|S^\pm|$ with the inequality $|S^\pm|\leq |\mathsf S^\pm|/(2\lambda-1)$ from Lemma~\ref{lem:induction}.
\end{proof}

\begin{theorem}
\label{thm:mainTechnicalTheorem}
    Fix an action of $F_S$ on an $n$-dimensional $\mathrm{CAT}(0)$ cube complex $X$. Let $x\in X^0$ be a $0$-cube such that $\Fix(x)$ is trivial. Suppose for some $\epsilon_1\in(0,1)$ that $|\B(H)|\leq\epsilon_1|S^\pm|$ for every $H\in\H$. There exists a progressing automaton with at most $(|S^\pm|n+1)^n$ vertices and $\theta(n,\epsilon_1)$-large growth where $\theta(n,\epsilon_1)>0$ depends only on $n$ and $\epsilon_1$. 
\end{theorem}

\begin{proof}
    Let $\epsilon(n)$ be the function from Lemma~\ref{lem:keyLemma}, and let $z=\min\{\frac15, \epsilon(n)\}$. Let $\epsilon_0=\sup\{x\in (0,z] : \frac{\epsilon_1}{(2(1-2\epsilon_0)-1)^n}\leq 1\}$, and let $\lambda=1-2\epsilon_0$. These definitions are arranged so that $\epsilon_0\in (0,\epsilon(n)]$ and $2\lambda-1\in (0,1)$ and $\frac{\epsilon_1}{(2\lambda-1)^{n-i}}<1$ for $1\leq i\leq n$.

    Define a positive decreasing sequence $\gamma(i)$, $1\leq i\leq n$, as follows:
    
    \begin{itemize}
        \item Let $\gamma(1)=1-\frac{\epsilon_1}{(2\lambda-1)^{n-1}}$.
        \item For $1<i\leq n$, let $\gamma(i)=\min\big\{\epsilon_0,\alpha\big(\epsilon_0,\frac{\epsilon_1}{(2\lambda-1)^{n-i}},i\big),(2\lambda-1)\gamma(i-1)\big\}$.
    \end{itemize}
    
    Note that the sequence depends on the constant $\epsilon_1$. We will prove the following claim by induction on $i$ for $1\leq i\leq n$. \smallskip
    
    \textit{Claim.} Suppose that no more than $i$ hyperplanes of $\H$ pairwise cross and $|\B(H)|\leq \frac{\epsilon_1}{(2\lambda-1)^{n-i}}|S^\pm|$ for each $H\in \H$. Then there exists a progressing automaton with $\gamma(i)$-large growth. \smallskip
    
    The theorem is a consequence of this claim. Indeed, defining $\theta(n,\epsilon_1)=\gamma(n)$ and applying the claim for the case $i=n$ proves the theorem. \smallskip
    
    \textit{Proof of Claim.} If no two hyperplanes in $\H$ cross, i.e., $i=1$, then $\P(H)$ is empty for every $H\in\H$. Thus $\A(H)\sqcup \B(H)$ is a partition of $S^\pm$ for every $H\in \H$. By assumption we have $\B(H)\leq \epsilon_1|S^\pm|$ for each $H\in \H$, so in particular, we have the weaker bound $\B(H)\leq \frac{\epsilon_1}{(2\lambda-1)^{n-1}}|S^\pm|$. It follows that $\A(H)\geq (1-\frac{\epsilon_1}{(2\lambda-1)^{n-1}})|S^\pm|$ for each $H\in \H$. By Corollary~\ref{cor:easyProgression}, we can build a progressing checkpoint tree with $(1-\frac{\epsilon_1}{(2\lambda-1)^{n-1}})$-large growth, i.e., $\gamma(1)$-large growth. By Lemma~\ref{lem:realization}, we can construct a progressing automaton with $\gamma(1)$-large growth. We now prove the claim holds for $i$ supposing it holds for~$i-1$.
    
    Suppose $|\B(H)|\leq \frac{\epsilon_1}{(2\lambda-1)^{n-i}}|S^\pm|$ and no more than $i$ hyperplanes of $\H$ pairwise cross. If $|\B(H)|\geq \epsilon_0|S^\pm|$, then there exists a progressing checkpoint tree at $H$ with $\alpha(\epsilon_0,\frac{\epsilon_1}{(2\lambda-1)^{n-i}},i)$-large growth by Lemma~\ref{lem:keyLemma}. Similarly, if $|\A(H)|\geq \epsilon_0|S^\pm|$, then there exists a progressing checkpoint tree at $H$ with $\epsilon_0$-large growth by Corollary~\ref{cor:easyProgression}. By definition, $\gamma(i)$ is no greater than the above two growth constants. Thus if $|\A(H)|\geq \epsilon_0|S^\pm|$ or $|\B(H)|\geq \epsilon_0|S^\pm|$ for every $H\in \H$, then there exists a progressing automaton with $\gamma(i)$-large growth by Lemma~\ref{lem:realization}.

    Suppose that $|\A(H)|<\epsilon_0|S^\pm|$ and $|\B(H)|<\epsilon_0|S^\pm|$ for some $H\in \H$. Then $|\P(H)|\geq(1-2\epsilon_0)|S^\pm|=\lambda|S^\pm|$ since $\A(H)\sqcup \B(H)\sqcup \P(H)$ is a partition of $S^\pm$. By Lemma~\ref{lem:induction}, there exists $\mathsf S^\pm\subset S^\pm$ with $|\mathsf S^\pm|\geq (2\lambda-1)|S^\pm|$ such that no more than $i-1$ hyperplanes of $\H_{\mathsf S^\pm}$ pairwise cross. Additionally, by Lemma~\ref{lem:backwardsInductiveBound}, $|\B(\mathsf H)|\leq \frac{\epsilon_1}{(2\lambda-1)^{n-(i-1)}}|\mathsf S^\pm|$ for each $\mathsf H\in \H_{\mathsf S^\pm}$. Restrict the action to the subgroup $F_{\mathsf S}\subset F_S$, and apply the $i-1$ version of the claim to $F_{\mathsf S}$. By the inductive hypothesis, we obtain a progressing automaton $\Sigma$ over $\mathsf S^\pm$ with $\gamma(i-1)$-large growth. Considered as an automaton over $S^\pm$, $\Sigma$ has $(2\lambda-1)\gamma(i-1)$-large growth. By definition, $\gamma(i)\leq (2\lambda-1)\gamma(i-1)$, so $\Sigma$ has $\gamma(i)$-large growth.
\end{proof}

\section{Random Groups and Cubical fixed-points}
\label{sec:mainTheorem}

\subsection{Random group preliminaries}
\begin{definition}[Random group $G_{S,d,L}$]
\label{def:randomGroup}
    If $S^\pm=S\sqcup S^{-1}$ is a set of formal letters and their formal inverses, then the set $\mathcal R_L$ of length $L$ words over $S^\pm$ has cardinality $|S^\pm|^L$. A \emph{random set of relators $R_{d,L}$ at density $d\in(0,1)$ and length $L$} is formed by taking $|S^\pm|^{dL}$ uniformly random samples from $\mathcal R_L$. A \emph{random group over $S$ at density $d\in (0,1)$ and length $L$} is given by a presentation $G_{S,d,L}=\langle S\mid R_{d,L}\rangle$. A random group at density $d$ has a property $Q$ \emph{with overwhelming probability} (\emph{w.o.p.}) if $\Prob\{G_{S,d,L}\text{ has }Q\}\to 1$ as $L\to \infty$.
\end{definition}

\begin{definition}[$A$-density]
    Let $A\subset \N$ be infinite. A set of words $\L$ over $S^\pm$ has $A$-\emph{density $d\in (0,1)$} if for some constant $c>0$, the set $\L$ contains at least $c|S^\pm|^{dL}$ words of length $L$ for each $L\in A$.  
\end{definition}

If an automaton $\Sigma$ has $\lambda$-large growth for $\lambda> |S^\pm|^{-d}$, then its accepted language $\L_\Sigma$ has $A$-density $d'>1-d$ for every $A\subset \N$. Indeed, we have $(\lambda|S^\pm|)^L>|S^\pm|^{(1-d)L}$, so $|S^\pm|^{d'L}$ is intermediate to the two terms for some $d'>1-d$. Note for a fixed $\lambda>0$, the inequality $\lambda> |S^\pm|^{-d}$ is satisfied for $|S^\pm|$ sufficiently large. The following lemma follows from an application of Chebyshev's theorem. For more details see \cite{gromov:asymptoticInvariants, orlef:randomNotOrderable}.

\begin{lemma}
\label{lem:denseLanguageIntersection}
    Fix densities $d,d'\in (0,1)$ so that $d+d'>1$. If $\L$ is a set of words with $A$-density $d'$, then for $L\in A$, $\Prob\{R_{d,L}\cap \L\neq\emptyset\}\to 1$ as $L\to \infty$. Consequently, if $\Sigma$ is an automaton with $\lambda$-large growth where $\lambda> |S^\pm|^{-d}$, then $\Prob\{R_{d,L}\cap \L_\Sigma\neq \emptyset\}\to 1$ as $L\to \infty$.
\end{lemma}

\begin{definition}[$H_m$, $\widehat S_m^\pm$, $\widehat S_m$]
\label{def:Hmsubgroup}
    For $m>0$ the subgroup $H_m<G_{S,d,L}$ is the subgroup generated by length $m$ words over $S^\pm$. Inversion induces a fixed-point free involution on the set of length $m$ words $\widehat S_m^\pm$. Thus for each $m>0$, we can fix a set of length $m$ words $\widehat S_m$ so that $\widehat S_m\sqcup \widehat S_m^{-1}$ is a partition of $\widehat S_m^\pm$.
\end{definition}

Note that $H_m$ has finite-index in $G_{S,d,L}$. Also note that $|\widehat S_m^\pm|=|S^\pm|^m$, so $|\widehat S_m^\pm|$ can be made arbitrarily large by increasing $m$. 

\begin{definition}[$\widehat G_m\twoheadrightarrow H_m$]
\label{def:mAssociatedGroup}
    For $m>0$, let $A_r(m)$ be the $r$-residue class modulo $m$. That is, $A_r(m)\subset \N$ is the set of $L$ such that $L=\widehat Lm+r$ with $r\in[0,m)$. 

    For $L\in A_r(m)$, if two relators of $G_{S,d,L}$ have the form $R_1w^{-1}$ and $wR_2$ for a word $w$ of length $r$, then $R_1R_2$ can be viewed as a word of length $2\widehat L$ over $\widehat S_m^\pm$. Let $\widehat R_{m,L}$ be the set of $R_1R_2$ for all such pairs $R_1w^{-1}$, $wR_2$. The group $\widehat G_m$ is defined by the presentation $\langle \widehat S_m \mid \widehat R_{m,L}\rangle$. Interpreting elements of $\widehat S_m$ as length $m$ words over $S^\pm$ defines a natural quotient map $\widehat G_m\twoheadrightarrow H_m$.
\end{definition}

\subsection{Cubical fixed-points}

\begin{lemma}
\label{lem:virtualFCn}
    Let $H< G$ be a finite-index subgroup. If $H$ has property $FW_n$, then $G$ has property $FW_n$. 
\end{lemma}

\begin{proof}
    Fix an action of $G$ on an $n$-dimensional $\mathrm{CAT}(0)$ cube complex $X$, and consider a global fixed-point of $H$, i.e., a point $x\in X$ such that $Hx=x$. The orbit $Gx$ is finite, since $H$ is finite-index. By \cite[Corollary~2.8]{bridsonHaefliger}, the fixed-point set of $G$ is non-empty. 
\end{proof}

Note that a global fixed-point is not necessarily a $0$-cube. Indeed, $\Z/2\Z$ acts by reflection on a $1$-cube fixing only the midpoint. The converse of Lemma~\ref{lem:virtualFCn} does not hold in general. For example, the 333-triangle group has $FW_1$ but is virtually $\Z^2$. 


\begin{lemma}
\label{lem:quotientFCn}
    If $G$ has property $FW_n$, then any quotient of $G$ has property $FW_n$.
\end{lemma}

\begin{proof}
    Let $G\twoheadrightarrow H$ be a group epimorphism. Any action $H\to \Isom(X)$ of $H$ on a $\mathrm{CAT}(0)$ cube complex $X$ can be turned into an action $G\twoheadrightarrow H\to \Isom(X)$ of $G$ on $X$. A global fixed-point of $G\to \Isom(X)$ is also a global fixed-point of $H\to \Isom(X)$.
\end{proof}

\begin{proof}[Proof of Lemma~\ref{lem:shortExactSequence}] 
    Let $X$ be a $\mathrm{CAT}(0)$ cube complex. For $0$-cubes $x$, $y$, and $z$, the \emph{median} $\mu(x,y,z)$ is the unique $0$-cube contained in the intersection of halfspaces containing a majority of $x$, $y$, and $z$. A collection of $0$-cubes $Y\subset X$ closed under taking medians is a \emph{median subalgebra}. Any median subalgebra $Y\subset X$ can be realized as the $0$-cubes of its own $\mathrm{CAT}(0)$ cube complex $X_Y$ in such a way that the medians of $X$ and $X_Y$ agree on $Y$. Additionally, $X_Y$ is uniquely determined by the median structure on $Y$, the dimension of $X_Y$ is no greater than the dimension of $X$, and any median preserving action on $Y$ extends to an action on $X_Y$. For more information, see \cite{nica:cubeSpacesWalls}. 

    Suppose groups $K$ and $N$ have $FW_n$, and suppose $G$ fits into a short exact sequence of the form $1\to K\to G\to N\to 1$. We show $G$ has $FW_n$. Let $G$ act on an $n$-dimensional $\mathrm{CAT}(0)$ cube complex $X$. This action induces an action of $K$ on $X$. By assumption, $K$ fixes a point $x\in X$ and thus stabilizes the open cube $C$ containing $x$. Thus after passing to the cubical subdivision of $X$, we can assume $K$ fixes a $0$-cube (e.g. the center of $C$). 
    Let $Y$ denote the set of $0$-cubes fixed by $K$, and observe that $Y$ is a median subalgebra. Since $K$ fixes $Y$ pointwise, there is an action of $N=G/K$ on $Y$. Extend this to an action of $N$ on $X_Y$. Note $X_Y$ is $n$-dimensional, since $X$ is $n$-dimensional. Since $N$ has $FW_n$, there exists a global fixed-point in $X_Y$. The cube containing this fixed-point is stabilized by $N$, so there exists $y\in Y$ with $Ny$ finite. Since $Ny=NKy=Gy$, the action of $G$ on $X$ has a finite orbit. By \cite[Corollary~2.8]{bridsonHaefliger}, there exists a global fixed-point of the $G$-action on $X$.
\end{proof}

\subsection{Main Theorem}

\begin{definition}[Language $\L_\Sigma^v$]
\label{def:prefix}
    Fix $m>0$. Let $\Sigma$ be an automaton over $\widehat S_m^\pm$, and let $v\in \Sigma$ be a vertex. We define $\mathcal L_\Sigma^v=\{w:w \text{ labels a directed path } P\to \Sigma \text{ with initial vertex } v\}$.
\end{definition}

\begin{remark}
\label{rem:differentDensities}
    Note that $\mathcal L_\Sigma=\mathcal L_\Sigma^{\mathfrak s}$ where $\mathfrak s\in \Sigma$ is the start vertex, so the accepted language is a special case of Definition~\ref{def:prefix}. For any $v\in \Sigma$ the language $\mathcal L_\Sigma^v$ is over $\widehat S_m^\pm$, but we can also view it as a language over $S^\pm$. Suppose $\Sigma$ has $\lambda$-large growth for $\lambda>|\widehat S^\pm_m|^{-d}$. Then $\mathcal L_\Sigma^v$ has density $d'>1-d$ as a language over $\widehat S_m^\pm$. Viewed as a language over $S^\pm$, the words of $\mathcal L_\Sigma^v$ all have length a multiple of $m$. Let $L\in A_0(m)$, and let $L=\widehat Lm$. By $\lambda$-large growth, $\mathcal L_\Sigma^v$ contains $c(\lambda|S_m^\pm|)^{\widehat L}$ words of length $\widehat L$ over $\widehat S_m^\pm$ for some fixed $c>0$. Since $\lambda>|\widehat S_m^\pm|^{-d}=|S^\pm|^{-md}$, we have $c(\lambda|\widehat S_m^\pm|)^{\widehat L}=c(\lambda^{1/m}|S^\pm|)^{\widehat Lm}>c(|S^\pm|^{d''})^L$ for some $d''>1-d$. This computation shows $\mathcal L_\Sigma^v$ has $A_0(m)$-density $d''>1-d$ as a language over $S^\pm$. As a consequence, for any $v\in \Sigma$ and for any word $w$ of length $r\in [0,m)$ over $S^\pm$, the languages $w\mathcal L_\Sigma^v$ and $\mathcal L_\Sigma^vw$ have $A_r(m)$-density $d''>1-d$ as languages over $S^\pm$.
\end{remark}

\begin{proof}[Proof of Theorem~\ref{thm:mainTheorem}]
    Set $\lambda=\min\{\frac13,\frac23\theta(n,\frac34)\}$ where $\theta(n,\epsilon_1)$ is the function from Theorem~\ref{thm:mainTechnicalTheorem}. Fix $S^\pm$ and $d\in(0,1)$, and take $m$ sufficiently large so that $\lambda>|S^\pm|^{-md}$. Let $\widehat S_m$ and $\widehat S_m^\pm$ be as in Definition~\ref{def:Hmsubgroup}, and note that $\lambda>|\widehat S_m^\pm|^{-d}$. Let $\mathbf \Sigma$ be the set of all automata over $\widehat S_m^\pm$ with $\lambda$-large growth and at most $(|\widehat S_m^\pm|n+1)^n$ vertices.

    We aim to show that $\widehat R_{m,L}\cap \mathcal L_\Sigma\neq \emptyset$ for every $\Sigma\in \mathbf \Sigma$ with overwhelming probability (w.o.p.), i.e., $\Prob\{\forall\Sigma\in\mathbf \Sigma(\widehat R_{m,L}\cap \mathcal L_\Sigma\neq \emptyset)\}\to 1$ as $L\to \infty$. Since $\mathbf\Sigma$ is finite, it suffices to prove the intersection is nonempty for a fixed $\Sigma\in \mathbf\Sigma$. It also suffices to prove the statement for $L\in A_r(m)$ for each $r\in[0,m)$.
    
    Fix $\Sigma\in \mathbf\Sigma$ and $r\in[0,m)$. It suffices to show that w.o.p. there exists a pair of the form $R_1w^{-1}, wR_2\in R_{d,L}$ such that $R_1R_2\in \mathcal L_\Sigma$, since by definition $R_1R_2\in \widehat R_{m,L}$. Fix a word $w$ of length $r$ over $S^\pm$. As discussed in Remark~\ref{rem:differentDensities}, for all $v\in \Sigma$, $\mathcal L_\Sigma^v w^{-1}$ and $w\mathcal L_\Sigma^v$ have $A_r(m)$-density $d'>1-d$ as languages over $S^\pm$. Thus by Lemma~\ref{lem:denseLanguageIntersection}, for $L\in A_r(m)$ and all $v\in \Sigma$, the intersections $R_{d,L}\cap \mathcal L_\Sigma^v w^{-1}$ and $R_{d,L}\cap w\mathcal L_\Sigma^v$ are nonempty w.o.p. Let $R_1w^{-1}\in R_{d,L}\cap \mathcal L_\Sigma^\s w^{-1}$, and let $u$ be the terminal vertex of the path $P\to \Sigma$ beginning at $\s$ with label $R_1$. Let $wR_2\in R_{d,L}\cap w\mathcal L_\Sigma^u$. Then by definition $R_1R_2$ is the label of a path $P\to \Sigma$ beginning at $\s$, i.e., $R_1R_2\in \mathcal L_\Sigma$. 

    For each $r\in[0,m)$ we have shown, for $L\in A_r(m)$, that $\Prob\{\forall\Sigma\in\mathbf \Sigma(\widehat R_{m,L}\cap \mathcal L_\Sigma\neq \emptyset)\}\to 1$ as $L\to \infty$. Thus the statement holds for all $L\in\mathbf N$. From this we deduce that $\widehat G_m$ has property $FW_n$ w.o.p.  If $\widehat G_m$ acts without global fixed-point on an $n$-dimensional $\mathrm{CAT}(0)$ cube complex $X$, then there is an induced action of $F_{\widehat S_m}$ on $X$ without global fixed-point. Fix a basepoint $x\in X^0$ minimizing $\sum_{s\in \widehat S_m^\pm}\dist(x,sx)$. If $|\mathrm{Fix}_{\widehat S_m}(x)|\geq \frac13|\widehat S_m^\pm|$, then there exists $\Sigma\in \mathbf\Sigma$ such that $wx\neq x$ for all $w\in\L_\Sigma$ by Lemma~\ref{lem:generatorsWithFixedPoint}. With overwhelming probability this leads to a contradiction, since $\mathcal L_\Sigma\cap \widehat R_{m,L}\neq \emptyset$ w.o.p. Suppose $|\Fix(x)|<\frac13|\widehat S_m^\pm|$. Then we have $|\mathsf S^\pm|\geq \frac23|\widehat S_m^\pm|$ where $\mathsf S^\pm=\widehat S_m^\pm-\mathrm{Fix}_{\widehat S_m}(x)$. For each $H\in \H$, $|\mathcal B(H)|\leq \frac12|\widehat S_m^\pm|$ by Lemma~\ref{lem:backwardsBound}, so $|\B(H)|\leq \frac34|\mathsf S^\pm|$. By Theorem~\ref{thm:mainTechnicalTheorem}, there exists a progressing automaton $\Sigma$ over $\mathsf S^\pm$ with at most $(|\mathsf S^\pm|n+1)^n$ vertices and $\theta(n,\frac34)$-large growth. As an automaton over $\widehat S_m^\pm$, $\Sigma$ has $\frac23\theta(n,\frac34)$-large growth. Since $\lambda\leq \frac23\theta(n,\frac34)$ and $(|\mathsf S^\pm|n+1)^n<(|\widehat S_m^\pm|n+1)^n$, we have $\Sigma\in \mathbf \Sigma$. Since $\Sigma$ is progressing, $wx\neq x$ for all $w\in \mathcal L_\Sigma$. As before, this leads to a contradiction w.o.p.

    If $\widehat G_m$ has $FW_n$, then $H_m$ has $FW_n$ by Lemma~\ref{lem:quotientFCn}. If $H_m$ has $FW_n$, then $G_{S,d,L}$ has $FW_n$ by Lemma~\ref{lem:virtualFCn}. Since $\widehat G_m$ has $FW_n$ w.o.p. so does $G_{S,d,L}$.
\end{proof}

\subsection{Constructing $FW_n$ groups}
We provide a method for constructing $FW_n$ groups by carefully choosing relators of a defining presentation. There is enough flexibility in the choice of relators that one can construct interesting examples. We hope the following will be a useful ``black-box'' for those wishing to construct examples of $FW_n$ groups.

\begin{theorem}
\label{thm:constructingExamples}
    Let $S^\pm$ be a finite set of formal letters and their inverses, and let $k=|S^\pm|$. For any $n>0$, there exists a finite set of automata $\mathbf\Sigma$ over $S^\pm$ so that the following holds: If $G$ is defined by a presentation $\langle S\mid R\rangle$ such that $R$ contains at least one word from $\L_\Sigma$ for each $\Sigma\in \mathbf \Sigma$, then $G$ has $FW_n$. 

    Moreover, there are constants $0<c_1, c_2<1$, depending only on $n$, so that each $\Sigma=(V,E)\in \mathbf \Sigma$ has one of the following two forms:
    \begin{itemize}
        \item $\Sigma$ has two vertices $V=\{\s,v\}$, two directed edges joining $\s$ to $v$ labeled $s$ and $s^{-1}$ for some $s\in S$, and $c_1|S^\pm|$ directed loops at $v$ with labels in $S^\pm-\{s^\pm\}$.
        \item The start vertex $\s$ has $(1-c_1)|S^\pm|$ outgoing edges, none of which are loops, and every vertex $v\in V-\{\s\}$ has $c_2|S^\pm|$ outgoing edges with terminal vertices in $V-\{\s\}$.
    \end{itemize}
    In particular, if $|S^\pm|$ is sufficiently large, then for any $p\geq 1$, one can choose words in $\L_\Sigma$ for each $\Sigma\in \mathbf\Sigma$, so that the set of chosen words satisfies the $C'(\frac{1}{p})$ condition.
\end{theorem}

\begin{proof}
    Fix some $0<c_1<1$. Take an action of $F_S$ on an $n$-dimensional $\mathrm{CAT}(0)$ cube complex without global fixed-point, and let $x\in X^0$ be a $0$-cube minimizing $\sum_{s\in S^\pm}\dist(x,sx)$. If $|\Fix(x)|\geq c_1|S^\pm|$, it follows from the proof of Lemma~\ref{lem:generatorsWithFixedPoint} that there exists an automaton $\Sigma$ of the first form so that $wx\neq x$ for each $w\in \L_\Sigma$. If $|\Fix(x)|<c_1 |S^\pm|$, then $\mathsf S^\pm\subset S^\pm-\Fix(x)$ is a symmetric set of generators, and the fix-set $\mathrm{Fix}_{\mathsf S}(x)$ of $F_{\mathsf S^\pm}$ is empty. For each $\mathsf H\in \H_{\mathsf S^\pm}$, we have $\B(\mathsf H)\leq \frac{1}{2(1-c_1)}|\mathsf S^\pm|$ by Lemma~\ref{lem:backwardsBound} and Lemma~\ref{lem:backwardsInductiveBound}. By Theorem~\ref{thm:mainTechnicalTheorem}, there exists a progressing automaton $\Sigma$ over $\mathsf S^\pm$ with $c_2=\theta(n,\frac{1}{2(1-c_1)})$-large growth and at most $(kn+1)^n$ vertices. Since $\Sigma$ is progressing, we have $wx\neq w$ for each $w\in \L_\Sigma$. Moreover, it can be checked that the automaton $\Sigma$ constructed in Theorem~\ref{thm:mainTechnicalTheorem} is of the second form. 

    Let $\mathbf\Sigma$ be the set of all automata with at most $(kn+1)^n$ vertices that are one of the two forms in the theorem. Let $G$ be a group defined by a presentation $\langle S\mid R\rangle$ such that $R$ contains at least one word from $\L_\Sigma$ for each $\Sigma\in\mathbf\Sigma$. Suppose $G$ acts without global fixed-point on an $n$-dimensional $\mathrm{CAT}(0)$ cube complex $X$. The action induces an action of $F_S$ on $X$ without global fixed-point. Let $x\in X^0$ be a $0$-cube minimizing $\sum_{s\in S^\pm}\dist(x,sx)$. Then there is some $\Sigma\in\mathbf\Sigma$ so that $wx\neq x$ for each $w\in\L_\Sigma$. This contradicts that $\L_\Sigma\cap R$ is non-empty.
\end{proof}

\subsection{Reduced words, the Rips construction, and an $FW_\infty$ cubulated group}\ 

Here we provide the proofs of Theorem~\ref{thm:reducedWords}, Theorem~\ref{thm:aMonster}, and Corollary~\ref{cor:ripsConstruction}, all stated in the introduction. The proof of Theorem~\ref{thm:reducedWords} is essentially the same as that of Theorem~\ref{thm:mainTheorem} but with one additional difficulty --- a reduced word over $\widehat S_m^\pm$ is not necessarily reduced when viewed as a word over $S^\pm$. The needed adjustment follows an idea similar to those in \cite[Lemma~3.4]{dahmaniguirardelprzytycki:randomDontSplit}, \cite[Lemma~5.5]{munro:randomCAT(0)SquareComplex}. We sketch the argument below. 

\begin{figure}[h]
    \centering
    \begin{subfigure}[h]{0.3\textwidth}
        \centering
        \includegraphics[width=\textwidth]{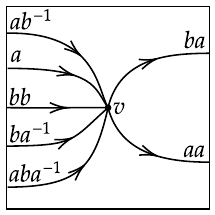}
        \caption{The original $v$.}	
        \label{}
    \end{subfigure}
	\ \ \ \ \ 
    \begin{subfigure}[h]{0.3\textwidth}
        \centering	
	\includegraphics[width=\textwidth]{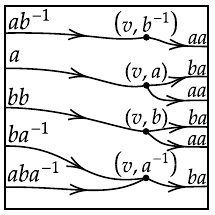}
        \caption{After ``rewiring''.}	
        \label{}
    \end{subfigure}
    \caption{A local picture of the replacement of $v$ with $\{(v,s)\}_{s\in S^\pm}$. In this example $S^\pm=\{a^\pm, b^\pm\}$.}
    \label{fig:rewire}
\end{figure}

\begin{proof}[Sketch for Theorem~\ref{thm:reducedWords}.]
    Replacing $R_{d,L}$, $\widehat S_m$, and $\widehat S_m^\pm$ by their reduced word analogues, the proof of Theorem~\ref{thm:mainTheorem} can mostly be repeated. Suppose $F_{\widehat S_m}$ acts without global fixed-point on an $n$-dimensional $\mathrm{CAT}(0)$ cube complex $X$, and $x\in X^0$ minimizes $\sum_{s\in \widehat S_m^\pm}\dist(x,sx)$. Lemma~\ref{lem:generatorsWithFixedPoint} and Theorem~\ref{thm:mainTechnicalTheorem} guarantee the existence of a progressing automaton $\Sigma$ over $\widehat S_m^\pm$ with $\lambda$-large growth such that $wx\neq x$ for each $w\in \mathcal L_\Sigma$. Moreover, $\lambda$ and the size of $\Sigma$ depend only on the dimension of $X$. 
    
    Let $\mathbf\Sigma$ be a finite collection of $\lambda$-large automata such that for any action of $F_{\widehat S_m}$ on an $n$-dimensional $\mathrm{CAT}(0)$ cube complex without global fixed-point, one can find $\Sigma\in \mathbf\Sigma$ as above. Elements of $\widehat S_m^\pm$ are length $m$ reduced words over $S^\pm$, and words in $\mathcal L_\Sigma$ are reduced words over $\widehat S_m^\pm$ for any $\Sigma\in \mathbf\Sigma$. However, $\mathcal L_\Sigma$ may contain unreduced words when viewed as a language over $S^\pm$. This can be remedied by adjusting the automata in $\mathbf \Sigma$, at the cost of increasing the number of vertices by a factor of $|S^\pm|$ and reducing the growth to $\lambda-\frac{1}{|S^\pm|}$.
    
    Let $\Sigma=(V,E)\in\mathbf\Sigma$, and let $v\in V-\{\s\}$ be a vertex with incoming edges $E_i$ and outgoing edges $E_o$. For $s\in S^\pm$, let $E_i(s)$ denote the incoming edges whose label is a word ending with $s$, and let $E_o(s)$ denote the outgoing edges whose label is a word not beginning with $s^{-1}$. Replace $v$ with a collection of vertices $\{(v,s)\}_{s\in S^\pm}$ and ``rewire'' the edges as follows. For each $s\in S^\pm$, change the incoming endpoints of edges in $E_i(s)$ to be $(v,s)$. Also for each $s\in S^\pm$, make a copy of $E_o(s)$ outgoing from $(v,s)$ with the same terminal vertices as the original $E_o(s)$. See Figure~\ref{fig:rewire}. Making these changes at each $v\in V$ produces an automaton $\Sigma'$ such that $\mathcal L_{\Sigma'}$ contains only reduced words when viewed as a language over $S^\pm$. Note that the fraction of words in $\widehat S_m^\pm$ beginning with a particular fixed letter is $\frac{1}{|S^\pm|}$. Thus for any vertex, if $|E_o|\geq \lambda|\widehat S_m^\pm|$, then $|E_o(s)|\geq \lambda-\frac{1}{|S^\pm|}$. Consequently each $\Sigma'$ has $(\lambda-\frac{1}{|S^\pm|})$-large growth since each $\Sigma\in \mathbf\Sigma$ has $\lambda$-large growth. If $\Sigma$ was such that $wx\neq x$ for each $w\in \mathcal L_\Sigma$, then $\mathcal L_{\Sigma'}$ has the same property since $\mathcal L_{\Sigma'}\subset \mathcal L_\Sigma$. Let $\mathbf \Sigma'$ be the set of all $\Sigma'$ for $\Sigma\in \mathbf \Sigma$.
    
    The remainder of the proof follows the same steps as in Theorem~\ref{thm:mainTheorem}, replacing $\mathbf\Sigma$ with $\mathbf\Sigma'$, provided that $|S^\pm|$ is sufficiently large to satisfy ${\lambda - \frac{1}{|S^\pm|}>0}$. We define $k(n) > 0$ to be large enough so that this inequality holds whenever $|S| > k(n)$. Under this assumption, we can choose $m$ so that $\lambda - \frac{1}{|S^\pm|} > |S^\pm|^{-md}$, and the proof proceeds as before.
\end{proof}

\begin{proof}[Proof of Theorem~\ref{thm:aMonster}]
    Removing a fixed finite number of relators from $R_{d,L}$ does not affect the density as $L\to \infty$. In particular, the conclusions of Corollary~\ref{cor:cubicalDimension} hold after removing a fixed finite number of relators. Thus, repeatedly applying Corollary~\ref{cor:cubicalDimension} produces a sequence of finite presentations $\{\langle S_i\mid R_i\cup \mathsf Q_i\rangle\}_{i\in \mathbf N}$ satisfying the following:
    
    \begin{enumerate}
        \item $|S_i|= k(i)$ where $k(i)$ is the function from Theorem~\ref{thm:reducedWords}. 
        \item $|\mathsf Q_i|=k(i-1)+k(i+1)$.
        \item Each $\langle S_i\mid R_i\cup \mathsf Q_i\rangle$ is torsion-free $C'(1/6)$ \cite{ollivier:invitationRandomGroups,gromov:asymptoticInvariants}.
        \item The group $G_i$ defined by $\langle S_i\mid R_i\rangle$ has $FW_i$.   
    \end{enumerate}

    The $R_i\cup\mathsf Q_i=R_{d,L}$ are random sets of reduced relators, and $\mathsf Q_i$ is a random selection of $k(i-1)+k(i+1)$ relators from $R_{d,L}$. Condition (4) is satisfiable because $R_i$ has the same density as $R_{d,L}$, as discussed above. We will need (3) just to know that $R_i$ contains no relators that are proper powers. 

    Let $B$ be $\mathbf N$ with its natural cell structure and with the presentation complex of $G_i$ wedge summed at each $i\in\mathbf N$. Let $\{s_{i,1},\ldots, s_{i,k(i)}\}$ be the $k(i)$ 1-cell loops wedge summed at $i$, corresponding to the generators of $G_i$, and let $t_i$ denote the 1-cell $[i,i+1]\subset \mathbf N$ of $B$. 
    
    Note that $\pi_1B=\ast_i G_i$. To obtain our desired group, we modify $B$ to obtain a $C'(1/6)$ complex $Z$ so that $\pi_1Z$ is a quotient of $G_i$ for each $i$. To do this, we add 2-cells so that each generator of $G_i$ at $i\in B$ becomes equal to a word in the generators of $G_{i-1}$ and also equal to a word in the generators of $G_{i+1}$. For each 1-cell loop $s_{i,j}$ at $i$, pick a 2-cell $l(i,j)$ in $\mathsf Q_{i-1}$ and a 2-cell $r(i,j)$ in $\mathsf Q_{i+1}$. Since $|\mathsf Q_i|=k(i-1)+k(i+1)$ for each $i$, this can be done so that the $l(i,j)$ and $r(i,j)$ are distinct for all $i\in \mathbf N$, $j\in \{1,\ldots, k(i)\}$. Let $\bar l(i,j)$ and $\bar r(i,j)$ denote the boundary paths of the $2$-cells $l(i,j)$ and $r(i,j)$. The complex $Z$ is obtained from $B$ by attaching 2-cells along the paths $s_{i,j}t_{i-1}^{-1}\bar l(i,j)t_{i-1}$ and $s_{i,j}t_i\bar r(i,j)t_i^{-1}$ for each $i\in \mathbf N$, $j\in \{1,\ldots, k(i)\}$. 

    Consider $s_{i,j}$ as an element of $G_i$. The paths $\bar l(i,j)$ and $\bar r(i,j)$ define elements of $G_{i-1}$ and $G_{i+1}$, and the $2$-cells attached to $B$ identify $s_{i,j}$ with $\bar l(i,j)$ and $\bar r(i,j)$. Thus we have $\pi_1 Z=\ast_iG_i/\llangle s_{i,j}=\bar l(i,j)=\bar r(i,j)\rrangle$. The equations $s_{i,j}=\bar l(i,j)=\bar r(i,j)$ show that each generator of $G_i$ can be written both as a word over the generators of $G_{i-1}$ and also as a word over the generators of $G_{i+1}$. Hence for any $i\in \mathbf N$, $\{s_{i,1},\ldots, s_{i,k(i)}\}$ generates $\pi_1 Z$ and is subject to all the relations in $G_i$. Consequently $\pi_1 Z$ is finitely generated and has $FW_n$ for all $n$. Since $Z$ is $C'(1/6)$ and no attaching map of a 2-cell is a proper power, $Z$ is aspherical. In particular $\pi_1 Z$ has geometric (and cohmological) dimension 2 \cite{lyndonschupp:combinatorialGrpTheory}. Since $Z$ is a locally finite aspherical $C'(1/6)$ complex, $\pi_1 Z$ acts freely on a locally finite $\mathrm{CAT}(0)$ cube complex \cite{wise:cubulatingC'(1/6)}.
\end{proof}

\begin{proof}[Proof of Corollary~\ref{cor:ripsConstruction}] 
    Let $G=\langle a_1,\ldots a_m\mid r_1,\ldots r_l\rangle$, and let $k(n)$ be the function from Theorem~\ref{thm:reducedWords}. Let $C=\max\{3, |r_1|,\ldots, |r_l|\}$. Let $\langle b_1,\ldots, b_{k(n)}\mid R_{d,L}\rangle$ be a random group in the reduced words model with $0<d<\frac{1}{2p}$. Randomly select $l$ words $\{B_1,\ldots, B_l\}$ and $m\cdot k(n)$ words $\{B_{i,j}\}$, $1\leq i\leq m$, $1\leq j\leq k(n)$ from $R_{d,L}$ without replacement, and let $\mathcal B$ be the set of remaining words. Since $\mathcal B$ and $R_{d,L}$ differ by a random uniformly bounded set of words, they have the same density. Thus by Corollary~\ref{cor:cubicalDimension}, we can find a presentation $\langle b_1,\ldots, b_{k(n)}\mid \mathcal R\rangle$ so that:
    \begin{enumerate}
        \item $\langle b_1,\ldots, b_{k(n)}\mid \mathcal R\rangle$ is $C'(1/2p)$.
        \item The relators in $\mathcal R$ have length $L>4pC$.
        \item $\mathcal R$ partitions into $\mathcal B$, $\{B_1,\ldots, B_l\}$, and $\{B_{i,j}\}_{1\leq i\leq m, 1\leq j\leq k(n)}$.
        \item Both $\langle b_1,\ldots, b_{k(n)}\mid \mathcal R\rangle$ and $\langle b_1,\ldots, b_{k(n)}\mid \mathcal B\rangle$ have $FW_n$.
    \end{enumerate}
    
    Let $\Gamma=\langle a_1,\ldots, a_m,b_1\ldots, b_{k(n)}\mid \mathcal B, r_1B_1,\ldots, r_lB_l, a_ib_ja_i^{-1}=B_{i,j}\rangle$. The subgroup $K=\langle b_1,\ldots, b_{k(n)}\rangle$ of $\Gamma$ is a is a quotient of $\langle b_1,\ldots, b_{k(n)}\mid \mathcal B\rangle$. Thus $K$ has $FW_n$ by Lemma~\ref{lem:quotientFCn}.
    
    The relators $a_ib_ja_i^{-1}=B_{i,j}$ imply $K$ is a normal subgroup of $\Gamma$, and it is easy to see that $\Gamma/K=G$. Thus $\Gamma$ has property $FW_n$ by Lemma~\ref{lem:shortExactSequence}. It remains to show that $\Gamma$ is $C'(1/p)$.
    
    Each relator of $\Gamma$ is a word in $\mathcal R$ concatenated with a word of length at most $C$. Thus if a word $w$ is a piece in $\Gamma$, then there exists a subword of $w$ of length $|w|-2C$ that is a piece in $\langle b_1,\ldots, b_{k(n)}\mid \mathcal R\rangle$. Let $w$ be a piece in $\Gamma$ that is $1/p$ the length of a relator. In particular, we have $|w|\geq \frac1p L$. From $L>4pC$ it follows that $\frac1p L-2C>\frac{1}{2p}L$. This inequality implies that there exists a subword $w'$ of $w$ that is a piece in $\langle b_1,\ldots, b_{k(n)}\mid \mathcal R\rangle$ and $1/2p$ the length of a relator. Since $\langle b_1,\ldots, b_{k(n)}\mid \mathcal R\rangle$ is $C'(1/2p)$, we conclude that $\Gamma$ must be $C'(1/p)$.
\end{proof}




\newpage
\bibliographystyle{alpha}
{\footnotesize
\bibliography{bibliography}}
\end{document}